\documentclass{article}

\usepackage{amsthm,amsmath,amssymb}
\usepackage[numbers]{natbib}

\usepackage{hyperref}
\hypersetup{
    colorlinks,%
    citecolor=black,%
    filecolor=black,%
    linkcolor=black,%
    urlcolor=black
}
\usepackage{hypernat}
\usepackage[top=2.5cm, bottom=2.5cm, left=2.8cm,
right=2.8cm]{geometry} 
\usepackage{verbatim}
\usepackage{graphicx}
\usepackage{caption}
\usepackage{subcaption}

\newtheorem{theorem}{Theorem}[section]

\newtheorem{lemma}[theorem]{Lemma}
\newtheorem{corollary}[theorem]{Corollary}
\newtheorem{remark}{Remark}[section]
\newtheorem{example}[theorem]{Example}

\newcommand{\Q}{{\mathbb Q}}

\newcommand{\R}{{\mathbb R}}

\newcommand{\C}{{\mathcal{C}}}

\newcommand{\jr}{{\mathcal{R}}}

\newcommand{\Ps}{{\mathcal{P}}}

\newcommand{\samp}{{\mathcal{X}}}

\newcommand{\bbC}{\mathbb{C}}

\newcounter{rcnt}[section]

\def\qt#1{\qquad\text{#1}}

\begin{document}

\title{Sharp Inequalities for $f$-divergences} 

\author{Adityanand Guntuboyina and Sujayam Saha and Geoffrey
  Schiebinger \\ University of California, Berkeley}

\maketitle

\begin{abstract}
\makebox{$f$-divergence}s are a general class of divergences between probability
measures which include as special cases many commonly used divergences
in probability, 
mathematical statistics and information theory such as
Kullback-Leibler divergence, chi-squared divergence, squared Hellinger
distance, total variation distance etc. In this
paper, we study the problem of
maximizing or minimizing an \makebox{$f$-divergence} between two probability
measures subject to a finite number of constraints on other
$f$-divergences. We show that these infinite-dimensional optimization
problems can all be reduced to optimization problems over small finite
dimensional spaces which are tractable. Our results lead to a
comprehensive and unified treatment of the problem of obtaining sharp
inequalities between \makebox{$f$-divergences}. We demonstrate that many of the
existing results on inequalities between $f$-divergences can be
obtained as special cases of our results and we also improve on some
existing non-sharp inequalities. 
\end{abstract}

\section{Introduction}
%
%
%
%

Suppose that the Kullback-Leibler divergence between
two probability measures is bounded from above by 2. What then is the
maximum possible value of  the Hellinger distance between them? Such
questions naturally arise in many fields including mathematical
statistics and machine learning, information theory, probability,
statistical physics etc. and the goal of this paper is to 
provide a way of answering them. From the variational viewpoint,
this problem can be posed as: maximize the Hellinger
distance subject to a constraint on the Kullback-Leibler divergence
over the space of all pairs of probability measures \textit{over all 
possible sample spaces}. We shall prove in this paper that the value of
this maximization problem remains unchanged if one restricts the
sample space to be the three-element set $\{1, 2, 3\}$. In
other words, in order to find the maximum Hellinger distance subject
to an upper bound on the Kullback-Leibler divergence, one can just
restrict attention to pairs of probability measures on $\{1, 2,
3\}$. Thus, the large infinite-dimensional optimization problem is
reduced to an optimization problem over a small finite-dimensional
space (of dimension $\leq 4$) which makes it tractable.  

In this paper, we prove such results in a very general
setting. The Kullback-Leibler divergence and the (square of the)
Hellinger distance are special instances of a general class of
divergences between probability measures called $f$-divergences (also
known as $\phi$-divergences). Let $f:(0,
\infty) \rightarrow \R$ be a convex function satisfying $f(1) = 0$. By
virtue of convexity, both the limits $f(0) := \lim_{x \downarrow 0}
f(x)$ and $f'(\infty) := \lim_{x \uparrow \infty} f(x)/x$ exist,
although they may equal $+\infty$. For two probability measures $P$
and $Q$, the $f$-divergence (see, for  example,~\cite{AliSilvey,
  Csiszar66, Csiszar67, Csiszar67fdiv}),
 $D_f(P||Q)$, is defined by      
\begin{equation*}
D_f(P||Q)  :=  \int_{q > 0} f\left( \frac{p}{q} \right) dQ +
f'(\infty)P\{q = 0\}
\end{equation*}
where $p$ and $q$ are densities of $P$ and $Q$ with respect to a
common measure $\lambda$. The definition does not depend on the choice
of the dominating measure $\lambda$. Special cases of $f$ lead to,
among others, Kullback-Leibler divergence, total variation distance,
square of the Hellinger distance and chi-squared divergence. 

We are now ready to introduce the general form of the optimization
problem we described at the beginning of the paper. Given divergences
$D_f$ and $D_{f_i}, i = 1, \dots, m$ and nonnegative real numbers
$D_1, \dots, D_m$, let  
\begin{equation*}
  A(D_1, \dots, D_m) := \sup \left\{D_f(P||Q) : D_{f_i}(P||Q)
    \leq D_i ~\forall  i\right\}
\end{equation*}
and 
\begin{equation*}
  B(D_1, \dots, D_m) := \inf \left\{D_f(P||Q) : D_{f_i}(P||Q)
    \geq D_i ~\forall  i\right\}
\end{equation*}
where the probability measures on the right hand sides above range
over all possible measurable spaces. The goal of this paper is to
provide a method for computing these quantities. We show that these
large infinite-dimensional optimization problems can all be reduced to 
optimization problems over small finite-dimensional
spaces. Specifically, in Theorem~\ref{main}, we show that in order to
compute these quantities, one can restrict attention to probability
measures on the set $\{1, \dots, m+2\}$.  

One of the main reasons for studying the quantities $A(D_1, \dots,
D_m)$  and $B(D_1, \dots, D_m)$ is that they yield sharp inequalities
for the divergence $D_f$ in terms of the divergences $D_{f_1},
\dots, D_{f_m}$. Indeed, the inequalities 
\begin{equation}\label{si}
  D_f(P||Q) \leq A(D_{f_1}(P||Q), \dots, D_{f_m}(P||Q))
\end{equation}
and 
\begin{equation}\label{sj}
  D_f(P||Q) \geq B(D_{f_1}(P||Q), \dots, D_{f_m}(P||Q))
\end{equation}
hold for every pair of probability measures $P$ and $Q$. Further, the
functions $A$ and $B$ satisfy the natural monotonicity inequalities
\begin{equation}\label{ma}
  A(D_1, \dots, D_m) \leq A(D_1', \dots, D_m')
\end{equation}
and 
\begin{equation}\label{mb}
  B(D_1, \dots, D_m) \leq B(D_1', \dots, D_m')
\end{equation}
for every $(D_1, \dots, D_m)$ and $(D_1', \dots, D_m')$ such that $D_i
\leq D_i'$ for all $i$.  

The inequalities~\eqref{si} and~\eqref{sj} are sharp in the sense that
$A$ is the \textit{smallest} function satisfying~\eqref{ma} for
which~\eqref{si} holds for all probability measures $P$ and
$Q$. Likewise, $B$ is the \textit{largest} function
satisfying~\eqref{mb} for which~\eqref{sj} 
holds for all probability measures $P$ and $Q$.

Inequalities between $f$-divergences are useful in many areas. For
example, in mathematical statistics, they are
crucial in problems of obtaining minimax bounds \cite{Yu97lecam,
  Tsybakovbook, GuntuThesis, GuntuFdiv}. In probability, such
inequalities are often used for converting limit theorems proved under
a convenient divergence into limit theorems for other divergences
\cite{BarronCLT,   Topsoe79, Harremoes}. They are also helpful for
proving results in measure concentration \cite{MartonAnnProb,
  MartonGFA,   MartonBlowUp}. Some applications in machine learning
are described in~\cite{ReidWilliamsonJMLR}. Further, inequalities
involving $f$-divergences are
fundamental to the field of information theory \cite{CoverThomas,
  CsiszarShields}.  

Because of their widespread use, many papers deal with inequalities 
between $f$-divergences (some references being~\cite{PinskerIneq,
  Csiszar66, Kullback67IEEE, kemperman69, Vajda70, GibbsSu,
  RefinePinsker, TopsoeCapDiv, Gilardoni06, GenPin,
  GuntuFdiv}). However, many of the inequalities presented in previous
treatments are not sharp. The few papers which
provide sharp inequalities~\cite{Vajda70, RefinePinsker, Gilardoni06,
  GenPin} only deal with certain special $f$-divergences 
as opposed to working in full generality. A popular such special case
is $m = 1$ and $D_{f_1}$ corresponding to the total variation
distance. In this case, sharp inequalities have been derived
in~\cite{RefinePinsker} for the case when $D_f$ is the 
Kullback-Leibler divergence and in~\cite{Gilardoni06} for the case of 
general $D_f$. The case $m > 1$ is comparatively less studied although
this has potential applications in the statistical problem of
obtaining lower bounds for the minimax risk (see Section~\ref{leca}
for details). The only paper which deals with sharp inequalities for
$m > 1$ is~\cite{GenPin} but there the authors only study the case
when $D_{f_1}, \dots, D_{f_m}$ are all primitive divergences (see
Remark~\ref{prim} below for the definition of primitive divergences).

In contrast with all previous papers in the area, we study the problem
of obtaining sharp inequalities between $f$-divergences in full
generality. In particular, our main results allow $m$ to be an
arbitrary positive integer and all
the divergences $D_f$ and $D_{f_1}, \dots, D_{f_m}$ to be arbitrary
$f$-divergences. We show that the underlying optimization problems can
all 
be reduced to low-dimensional optimization problems and we outline
methods for solving them. We also show that many of the existing
results on inequalities between $f$-divergences can be obtained as
special cases of our results and we also improve on some existing
non-sharp inequalities.  
  
The rest of this paper is structured as follows. Our main result is
stated in Theorem~\ref{main}. Its three-part proof is given
in Section~\ref{pruf}. The first part is based on a recent
representation theorem for 
$f$-divergences which implies that the optimization problems 
for computing $A(D_1, \dots, D_m)$ and $B(D_1, \dots, D_m)$ can be
thought of as maximizing or minimizing an integral functional over a
certain class of concave functions satisfying a finite number of
integral constraints. In the second part of the proof, we use
Choquet's theorem to restrict attention only to the extreme points of
the constraint set. Finally, in the third part, we characterize these
extreme points and show that they correspond to probability measures
over small finite sets.   

One possible approach to compute $A(D_1, \dots, D_m)$ and $B(D_1,
\dots, D_m)$ is via joint ranges of $f$-divergences. Specifically, for
$m \geq 1$ and divegences $D_{f_1}, \dots, D_{f_m}$, their joint
range, denoted by $\jr(f_1, \dots, f_m)$ is defined as the set of all
vectors in $\R^m$ that equal $(D_{f_1}(P||Q), \dots, D_{f_m}(P||Q))$
for some pair of probability measures $P$ and $Q$. If the joint range
$\jr(f, f_1, \dots, f_m)$ can be determined, then one can easily
calculate the values $A(D_1, \dots, D_m)$ and $B(D_1, \dots, D_m)$ for
every $D_1, \dots, D_m$. The problem of determining the joint range
$\jr(f_1, \dots, f_m)$ was solved for the case $m = 2$
in~\cite{HarremoesVajda}. We extend their result to general $m
\geq 2$ in Section~\ref{hash} by a very simple proof which was
communicated to us by an anonymous referee. Unfortunately, it turns
out that this approach based on the joint range does not quite prove
Theorem~\ref{main}. It gives a slightly weaker result. We discuss this
in Section~\ref{hash}.  

Also in Section~\ref{remex}, we collect some remarks and extensions of
our main theorem and, in particular, we show that the theorem is tight
in general. In Section~\ref{apsp}, we consider various special cases and
show that many well-known results in the literature can be obtained as
simple instances of our main theorem. In
Section~\ref{nucom}, we describe numerical methods for solving the
low-dimensional optimization problems that come out of our main
theorem. We solve an important subclass of these problems by
convex optimization and we also describe heuristic methods for the
general case. 

\section{Main Result}
For each $n \geq 1$, let $\Ps_n$ denote the space of all probability
measures defined on the finite set $\{1, \dots, n\}$. Let us define
$A_n(D_1, \dots, D_m)$ to be 
\begin{equation*}
  \sup \left\{D_f(P||Q): P, Q \in \Ps_n \text{ and } D_{f_i}(P||Q)
    \leq D_i ~\forall i \right\}
\end{equation*}
and, analogously, $B_n(D_1, \dots, D_m)$ to be 
\begin{equation*}
  \inf \left\{D_f(P||Q): P, Q \in \Ps_n \text{ and } D_{f_i}(P||Q)
    \geq D_i ~\forall i \right\}. 
\end{equation*}
Our main theorem is given below. The second part of the theorem
requires that $D_{f_1}, \dots, D_{f_m}$ are finite divergences. We say
that a divergence $D_f$ is finite if $\sup_{P, Q} D_f(P||Q) <
\infty$. The supremum here is taken over all probability measures over
all possible measurable spaces. See Remark~\ref{remfin} for a detailed
explanation of finite divergences. 
\begin{theorem}\label{main}
For every $D_1, \dots, D_m \geq 0$, we have 
  \begin{equation}\label{main.1}
    A(D_1, \dots, D_m) = A_{m+2}(D_1, \dots, D_m). 
  \end{equation}
Further if $D_{f_1}, \dots, D_{f_m}$ are all finite, then 
  \begin{equation}\label{main.2}
    B(D_1, \dots, D_m) = B_{m+2}(D_1, \dots, D_m). 
  \end{equation}
\end{theorem}
The conclusions of the above theorem may be better appreciated in the
following optimization form. Theorem~\ref{main} states that the
quantity $A(D_1, \dots, D_m)$ equals the optimal value of the
following finite-dimensional optimization problem: 
\begin{equation}\label{finiteA}
\begin{aligned}
& \underset{p,q \in [0,1]^{m+2}}{\textnormal{maximize}}
& & \sum_{j: q_j > 0} q_j f\left(\frac {p_j} {q_j} \right) +f'(\infty)
\sum_{j:q_j = 0} p_j \\ 
& \textnormal{subject to}
& & p_j \ge 0, \; q_j \ge 0 \textnormal{ for all } j = 1, \dots, m+2\\
& & & \sum p_j  = \sum q_j = 1\\
& & & \sum_{j: q_j > 0} q_j f_i \left(\frac {p_j} {q_j} \right) + f_i '
(\infty) \sum_{j:q_j = 0} p_j \le D_i
\end{aligned}
\end{equation}
for $i = 1, \dots, m$. Similarly, when $D_{f_1}, \dots, D_{f_m}$ are
all finite, $B(D_1, \dots, D_m)$ equals the optimal value of 
\begin{equation}\label{finiteB}
\begin{aligned}
& \underset{p,q \in [0,1]^{m+2}}{\textnormal{minimize}}
& & \sum_{j: q_j > 0} q_j f\left(\frac {p_j} {q_j} \right) +f'(\infty)
\sum_{j:q_j = 0} p_j \\ 
& \textnormal{subject to}
& & p_j\ge0, \; q_j\ge 0 \textnormal{ for all } j = 1, \dots, m+2 \\
& & & \sum p_j  = \sum q_j = 1\\
& & & \sum_{j: q_j > 0} q_j f_i \left(\frac {p_j} {q_j} \right) + f_i '
(\infty) \sum_{j:q_j = 0} p_j \ge D_i
\end{aligned}
\end{equation}
for $i = 1, \dots, m$. The proof of Theorem~\ref{main} is provided in
the next section. In Section~\ref{remex}, we argue that 
Theorem~\ref{main} is tight in general and also comment on the
assumption of finiteness of $D_{f_1}, \dots, D_{f_m}$ for the validity
of identity~\eqref{main.2}. We also describe an attempt to prove this
theorem via joint ranges but this only yields a weaker result. 

\section{Proof of the Main Result}\label{pruf}
\subsection{Testing Representation}\label{trep}
For two probability measures $P$ and $Q$, let us define the
function $\psi_{P, Q}: [0, \infty) \rightarrow [0, 1]$ by 
\begin{equation*}
  \psi_{P, Q}(s) := \int \min(p, qs) d\lambda \qt{for $s \in [0, \infty)$}
\end{equation*}
where $p$ and $q$ denote the densities of $P$ and $Q$ with respect to
a common measure $\lambda$ (which can, for example, be taken to be $P
+ Q$). This function $\psi_{P, Q}$ is nonnegative, concave,
non-decreasing and satisfies the inequality $0 \leq \psi_{P, Q}(s)
\leq \min(1, s)$ for all $s \geq 0$. In other words, $\psi \in \C$
where $\C$ denotes the class of all functions $\psi$ on $[0, \infty)$
that are nonnegative, concave, non-decreasing and satisfy the
inequality $\psi(s) \leq \min(1, s)$ for all $s \geq 0$. Moreover, it
is true (see, for example,~\cite[Corollary 5]{GenPin}) that
every function $\psi \in \C$ equals $\psi_{P, Q}$ for some pair of
probability measures $P$ and $Q$.  

For each divergence $D_f$, let us associate the measure $\nu_f$ on
$(0, \infty)$ defined by 
\begin{equation*}
  \nu_f(a, b] := \partial^r f(b) - \partial^r f(a) \qt{for $0 < a < b
    < \infty$} 
\end{equation*}
where $\partial^r$ denotes the right derivative operator (note that by
convexity $\partial^r f(x)$ exists for every $x \in (0, \infty)$). We 
also associate the functional $I_f: \C \rightarrow [0, \infty]$ by 
\begin{equation}\label{lie}
 I_f(\psi) := \int_0^{\infty} \left( \min(1, s) - \psi(s) \right)
 d\nu_f(s). 
\end{equation}
There is a precise connection between $D_f$ and $I_f$ that is given
below: 
\begin{lemma}\label{lv}
  For every pair of probability measures $P$ and $Q$, we have
  \begin{equation}\label{lv.eq}
    D_f(P||Q) = I_f(\psi_{P, Q}). 
  \end{equation}
\end{lemma}

Lemma~\ref{lv} is not new although the form in which it is stated
above is non-standard. The more standard version simply involves
writing the integral in~\eqref{lie} over the interval $(0, 1)$ by the
change of variable $t = s/(1+s)$. In this modified form,
Lemma~\ref{lv} has been proved in~\cite{OsterreicherVajda} in the case
when $f$ is twice differentiable and in~\cite{LieseVajda} in the
general case. A short proof is available in~\cite[Theorem 2.3]{Liese}.  

\begin{remark}[Primitive $f$-divergences]\label{prim}
For each $s > 0$, let $u_s(t) := \min(1,s) - \min(t, s)$ for $t \in
(0, \infty)$. Clearly, $u_s$ is a convex function on $(0, \infty)$
such that $u_s(1) = 0$. Moreover, it is a very simple convex function
in the sense that it is piecewise linear with just two linear
parts. It is straightforward to check that the divergence
corresponding to $u_s$ is given by:  
\begin{equation*}
  D_{u_s}(P||Q) = \min(1, s) - \psi_{P, Q}(s). 
\end{equation*}
Lemma~\ref{lv} therefore asserts that any arbitrary $f$-divergence can
be written as an integral of the primitive divergences $D_{u_s}$ with
respect to the measure $\nu_f$ on $(0, \infty)$. The most well-known
of these primitive divergences is the total variation distance which
corresponds to $s = 1$. Indeed, 
\begin{equation*}
  D_{u_1}(P||Q) = 1 - \int \min(p, q) d\lambda = \frac{1}{2}\int |p-q|
  d\lambda =: V(P, Q) 
\end{equation*}
Every primitive divergence $D_{u_s}(P||Q)$ is closely related to the
smallest weighted average error (Bayes risk) in the problem of
statistical testing between the hypotheses $P$ against $Q$ based on an
observation $X$ (see, for example,~\cite[Lemma 3]{GenPin}). 
\end{remark}

\begin{remark}[Finiteness of a divergence]\label{remfin}
Lemma~\ref{lv} imples that 
\begin{equation}\label{fdi}
  \sup_{P, Q} D_f(P||Q) = \int_0^{\infty} \min(1, s) d\nu_f(s) = f(0)
  + f'(\infty).
\end{equation}
The supremum above is taken over all probability measures $P$
and $Q$ defined on \textit{all possible} measurable spaces. To
see~\eqref{fdi}, just note that, by Lemma~\ref{lv}, we have
  \begin{equation*}
    \sup_{P, Q} D_f(P||Q) = \sup_{P, Q} I_f(\psi_{P, Q}) = \sup_{\psi
      \in \C} I_f(\psi) = I_f(0). 
  \end{equation*}
Intuitively, $\psi_{P, Q}(s) = 0$ for all $s$ implies that $P$ and $Q$
are maximally separated (mutually singular) and thus the maximum value
of $I_f(\psi)$ is achieved when $\psi$ is the identically zero
function. The definition of $I_f$ gives that 
\begin{equation*}
  I_f(0) = \int_0^{\infty} \min(1, s) d\nu_f(s) 
\end{equation*}
Moreover, for the probability measures $P^* = (1, 0)$ and $Q^*
= (0, 1)$ in $\Ps_2$, the function $\psi_{P, Q}$ equals 0. Therefore, 
\begin{equation*}
I_f(0) = D_f(P^*||Q^*) = f(0) + f'(\infty), 
\end{equation*}
which proves~\eqref{fdi}. 

Recall that an $f$-divergence is \textit{finite} if $\sup_{P, Q}
D_f(P||Q) < \infty$. By~\eqref{fdi}, an $f$-divergence is finite if
and only if 
\begin{equation}\label{finchar}
  \int_0^{\infty} \min(1, s) d\nu_f(s) = f(0) + f'(\infty) < \infty. 
\end{equation}
Well known examples of finite divergences are the primitive
divergences, the square of the Hellinger distance and the capacitory
discrimination (which corresponds to the convex function~\eqref{capd}).  
\end{remark}

For each $f$ and $D \geq 0$, let us define 
\begin{equation*}
  \C_1(f, D) := \left\{\psi \in \C: I_f(\psi) \leq D \right\}
\end{equation*}
and
\begin{equation*}
  \C_2(f, D) := \left\{\psi \in \C: I_f(\psi) \geq D \right\}
\end{equation*}
As a consequence of Lemma~\ref{lv}, we obtain that 
\begin{equation}\label{s1a}
  A(D_1, \dots, D_m) = \sup \left\{I_f(\psi) : \psi \in \cap_{i=1}^m
    \C_1(f_i, D_i) \right\}
\end{equation}
and 
\begin{equation}\label{s1b}
  B(D_1, \dots, D_m) = \inf \left\{I_f(\psi) : \psi \in \cap_{i=1}^m
    \C_2(f_i, D_i) \right\}. 
\end{equation}
The following lemma on the derivatives of the function $\psi_{P, Q}$
(the left and right derivative operators are denoted by $\partial^l$
and $\partial^r$ respectively) will be useful in the sequel. 
\begin{lemma}\label{der}
For every function $\psi = \psi_{P, Q}$ in $\C$, we have
\begin{equation}\label{pl}
  \partial^l \psi(s) = Q\left\{p \geq s q \right\} \qt{for $s > 0$}
\end{equation}
and 
\begin{equation}\label{pr}
\partial^r \psi(s) = Q \left\{p > s q \right\} \qt{for $s \geq 0$}. 
\end{equation}
\end{lemma}
\begin{proof}
For every $s > 0$,
\begin{equation*}
  \partial^l \psi(s) = \lim_{\epsilon \downarrow 0} \frac{\psi(s) -
    \psi(s-\epsilon)}{\epsilon}
\end{equation*}
and 
\begin{equation*}
  \frac{\psi(s) - \psi(s-\epsilon)}{\epsilon} = \int \frac{\min(p, qs)
  - \min(p, q(s-\epsilon))}{\epsilon} d\lambda
\end{equation*}
It is easy to check that the integrand above is bounded in absolute
value by $q$ and converges as $\epsilon \downarrow 0$ to $q \left\{p
  \geq q s \right\}$. The identity~\eqref{pl} therefore follows by the
dominated convergence theorem. The proof of~\eqref{pr} is similar. 
\end{proof}

\subsection{Reduction to Extreme Points}
Let us first recall the definition of extreme points. Let $S$ be a
subset of a vector space $V$. A point $a \in S$ is called an extreme
point of $S$ if $a = (b+c)/2$ for $b, c \in S$ implies that $a = b =
c$. In other words, $a$ cannot be the mid-point of a non-trivial line
segment whose end points lie in $S$. We denote the set of all extreme
points of $S$ by $ext(S)$. 

An important result about extreme points in infinite dimensional
topological vector spaces is Choquet's theorem (see,
for example,~\cite[Chapter 3]{Phelps66Choquet}). We shall use the
following version of Choquet's theorem in this section:
\begin{theorem}[Choquet]\label{lechoc}
  Let $K$ be a metrizable, compact convex subset of a locally convex
  space $V$ and let $x_0$ be an element of $K$. Then there exists a
  Borel probability measure $\mu_0$ on $K$ which is concentrated on
  the extreme points of $K$ and which satisfies $L(x_0) = \int_K L(x)
  d\mu_0(x)$ for every continuous linear functional $L$ on $V$. 
\end{theorem}

The goal of this section is to prove the following: 
\begin{lemma}\label{redex}
  For every $D_1, \dots, D_m \geq 0$, we have 
  \begin{equation*}
    A(D_1, \dots, D_m) = \sup \left\{I_f(\psi) : \psi \in ext
      \left(\cap_{i=1}^m \C_1(f_i, D_i) \right) \right\} 
  \end{equation*}
and further, if $D_{f_1}, \dots, D_{f_m}$ are all finite, we have 
  \begin{equation*}
    B(D_1, \dots, D_m) = \inf \left\{I_f(\psi) : \psi \in ext
      \left(\cap_{i=1}^m \C_2(f_i, D_i) \right) \right\} 
  \end{equation*}
\end{lemma}
\begin{proof}
  The proof is based on Theorem~\ref{lechoc}. Let $C[0, \infty)$
  denote the space of all continuous functions on $[0, \infty)$
  equipped with the topology given by the metric: 
  \begin{equation}\label{met}
    \rho(f, g) := \sum_{k \geq 1} 2^{-k} \min \left(\sup_{0 \leq x
        \leq k} |f(x) - g(x)|, 1 \right). 
  \end{equation}
 It is a fact (see, for example,~\cite[Chapter 1]{Rudin.func})
 that $C[0, \infty)$ is a locally convex vector space under this
 topology. We shall apply Choquet's theorem to $V = C[0, \infty)$ and
 $K = \cap_{i=1}^m \C_1(f_i, D_i)$ for the first identity and $K =
 \cap_{i=1}^m \C_2(f_i, D_i)$ for the second identity. It is obvious
 that $\C$ is a subset of $C[0, \infty)$. 

 Clearly both the sets $\cap_i \C_1(f_i, D_i)$ and $\cap_i \C_2(f_i, D_i)$
 are convex. Also, by Fatou's lemma, $\cap_i \C_1(f_i, D_i)$ is closed
 under pointwise convergence i.e., if $\psi_n \in \cap_i \C_1(f_i,
 D_i)$ and $\psi_n \rightarrow \psi$ pointwise, then $\psi \in \cap_i
 \C_1(f_i, D_i)$. To see this, observe that by Fatou's lemma, for each
 $i = 1, \dots, m$, 
 \begin{align*}
   I_{f_i}(\psi) &= \int_0^{\infty} \left(\min(1, s) - \psi(s) \right)
   d\nu_{f_i}(s) \\
&= \int_0^{\infty} \liminf_{n \rightarrow \infty} \left(\min(1, s) -
  \psi_n(s) \right) d\nu_{f_i}(s) \\
&\leq  \liminf_{n \rightarrow \infty} \int_0^{\infty} \left( \min(1,
  s) - \psi_n(s) \right) d\nu_{f_i}(s) \leq D_i. 
 \end{align*}
On the other hand, if each $D_{f_i}$ is a finite
 divergence, then by the dominated convergence theorem, $\cap_i
 \C_2(f_i, D_i)$ is also closed under pointwise convergence. Indeed,
 if $\psi_n \rightarrow \psi$ pointwise and $D_{f_i}$ is a finite
 divergence, then by the dominated convergence (since $0 \leq
 \min(1, s) - \psi_n(s) \leq \min(1, s)$), we have $I_{f_i}(\psi_n)
 \rightarrow I_{f_i}(\psi)$. 

In Lemma~\ref{com} below, we show that $\C$ is a compact subset of 
$C[0, \infty)$ under the topology given by the metric
$\rho$. Moreover, it is easy to see that convergence in the metric
$\rho$ implies pointwise convergence. It follows hence that $\cap_i
\C_1(f_i, D_i)$ is a compact, convex subset of $C[0, \infty)$ and if
each $D_{f_i}$ is a finite divergence, then $\cap_i \C_2(f_i, D_i)$ is
also a compact, convex subset of $C[0,\infty)$. 

For each $\epsilon > 0$, let us define the functional
$\Lambda_{\epsilon}$ on $C[0, \infty)$ by 
\begin{equation*}
  \Lambda_{\epsilon}(\psi) = \int \left(\min(1, s) - \psi(s) \right)
  \left\{\epsilon \leq s \leq 1/\epsilon \right\} d\nu_f(s)
\end{equation*}
When restricted to the interval $[\epsilon, 1/\epsilon]$, the measure
$\nu_f$ is a finite measure. Hence, $\Lambda_{\epsilon}$ is a
continuous, linear functional on $C[0, \infty)$. Thus, by
Theorem~\ref{lechoc}, we get that for every $\psi_0 \in \cap_i
\C_1(f_i, D_i)$, there exists a Borel probability measure $\tau_0$
that is concentrated on the set of extreme points, $ext(\cap_i
\C_1(f_i, D_i))$, of $\cap_i \C_1(f_i, D_i)$ such that 
\begin{equation*}
  \Lambda_{\epsilon}(\psi_0) = \int \Lambda_{\epsilon}(\psi) d\tau_0(\psi),
\end{equation*}
for every $\epsilon > 0$. Now, by the monotone convergence theorem, 
\begin{equation*}
  \Lambda_{\epsilon}(\psi) \uparrow I_f(\psi) \qt{as $\epsilon
    \downarrow 0$}
\end{equation*}
for every $\psi \in \C$. As a result, we can use the monotone
convergence theorem again to assert that
\begin{equation*}
  \int \Lambda_{\epsilon}(\psi) d\tau_0(\psi) \uparrow \int
  I_f(\psi) d\tau_0(\psi) \qt{as $\epsilon \downarrow 0$}. 
\end{equation*}
We therefore obtain
\begin{equation}\label{ifcon}
  I_f(\psi_0) = \int I_f(\psi) d\tau_0(\psi). 
\end{equation}
Since this is true for all functions $\psi_0$ in $\cap_i \C_1(f_i,
D_i)$, we obtain 
\begin{equation*}
  \sup_{\psi \in \cap_i \C_1(f_i, D_i)} I_f(\psi) =   \sup_{\psi \in
    ext \left(\cap_i \C_1(f_i, D_i)\right)} I_f(\psi) 
\end{equation*}
The proof of the first assertion of Lemma~\ref{redex} is now complete
by~\eqref{s1a}. Similarly, when each divergence $D_{f_i}$ is
finite, we can prove that 
\begin{equation*}
  \inf_{\psi \in \cap_i \C_2(f_i, D_i)} I_f(\psi) =   \inf_{\psi \in
    ext \left(\cap_i \C_2(f_i, D_i)\right)} I_f(\psi) 
\end{equation*}
and this, together with~\eqref{s1b}, completes the proof of
Lemma~\ref{redex}. 
\end{proof}

In the above proof, we used the fact that $\C$ is
compact in $C[0, \infty)$, the space of all continuous functions on 
$[0, \infty)$. We prove this fact below. 
\begin{lemma}\label{com}
  The class $\C$ is compact in $C[0, \infty)$ equipped with the
  topology given by the metric~\eqref{met}. 
\end{lemma}
\begin{proof}
  We show that $\C$ is sequentially compact. Consider a sequence
  $\{\psi_n\}$ in $\C$. For every fixed $s_0 \in [0, \infty)$, the
  sequence $\{\psi_n(s_0)\}$ is a sequence of real numbers in $[0, 1]$
  and hence has a convergent subsequence. By a standard
  diagonalization argument, we assert the existence of a subsequence
  $\{\phi_k\}$ of $\{\psi_n\}$ that converges pointwise over the
  set of all nonnegative rational numbers (denoted by $\Q_+$). 

  Let us now fix $\epsilon > 0$ and a real number $s_0 \in [0,
  \infty)$. Choose $r_1, r_2 \in \Q_+$ such that $r_1 \leq s_0 \leq
  r_2$ and such that $r_2 - r_1 < \epsilon/4$. Also, let $N \geq 1$
  be large enough so that 
  \begin{equation*}
    |\phi_k(r_i) - \phi_l(r_i)| < \epsilon/4 \qt{for $k, l \geq N$}
  \end{equation*}
  and for $i = 1, 2$. Using properties of functions in $\C$, we get
  that 
  \begin{align*}
    |\phi_k(s_0) - \phi_l(s_0)| &< |\phi_k(r_1) - \phi_l(r_2)| + |\phi_k(r_2) -
    \phi_l(r_1)| \\
&< 2 |\phi_k(r_1) - \phi_l(r_1)| + 2|r_1 - r_2| < \epsilon. 
  \end{align*}
  In the last inequality above, we have used the fact that functions
  in $\C$ are Lipschitz with constant 1 (this can be proved for
  instance using the derviatives given by Lemma~\ref{der}). It
  therefore follows that the sequence $\{\phi_k\}$ converges pointwise on
  $[0, \infty)$. The proof is now complete by the observation that 
  convergence in the metric $\rho$ is equivalent to pointwise
  convergence on $[0, \infty)$. 
\end{proof}

\subsection{Characterization of Extreme Points}
Lemma~\ref{redex} asserts that for the purposes of finding the
supremum or infimum of $I_f$ subject to constraints on $I_{f_i}$, it
is enough to focus on the extreme points of the constraint set. In the
next theorem, we provide a necessary condition for a function in the
constraint set to be an extreme point of the constraint set. 

\begin{theorem}\label{charex}
  Let $\psi$ be a function in $\cap_i \C_1(f_i,
  D_i)$ and let $k$ be the number of indices $i$ for which
  $I_{f_i}(\psi) = D_i$. Then a necessary condition for $\psi$ to be
  extreme in $\cap_i \C_1(f_i, D_i)$ is that $\psi$ equals $\psi_{P,
    Q}$ for two probability measures $P, Q \in \Ps_{k+2}$. The same
  conclusion also holds for extreme functions in $\cap_i \C_2(f_i,
  D_i)$ provided all the involved divergences $D_{f_1}, \dots,
  D_{f_m}$ are finite.  
\end{theorem}
\begin{remark}\label{cal}
When $m = k = 0$, the sets $\cap_i \C_1(f_i, D_i)$ and
$\cap_i\C_2(f_i, D_i)$ can both be taken to be equal to $\C$. As will
be clear from the 
proof, the above theorem will also be true in this case where it
states that a necessary condition for a function $\psi$  to be extreme
in $\C$ is that $\psi$ equals $\psi_{P, Q}$ for two probability
measures $P, Q \in \Ps_2$. 
\end{remark}

The proof of Theorem \ref{charex} relies on the following lemma whose proof is provided after the proof of Theorem~\ref{charex}.
\begin{lemma}\label{ss}
  Let $P$ and $Q$ be two probability measures on a space $\samp$
  having densities $p$ and $q$ with respect to $\lambda$. Let $l \geq
  1$ be fixed. Suppose that for every decreasing sequence $s_1 >
  \dots > s_l$ of positive real numbers, the following condition
  holds: 
  \begin{equation*}
    \min_{1 \leq j \leq l+1} \left(P(B_j) + Q(B_j) \right) = 0
  \end{equation*}
where $B_1 = \{p \geq q s_1 \}, B_i = \{q s_i \leq p < q s_{i-1} \}$
for $i = 2, \dots, l$ and $B_{l+1} = \{p < q s_{l} \}$. Then
$\psi_{P, Q}$ can be written as $\psi_{P', Q'}$ for two probability
measures $P', Q' \in \Ps_l$. 
\end{lemma}

\begin{proof}[Proof of Theorem~\ref{charex}]
Let $\psi$ be a function in $ext(\cap_i \C_1(f_i, D_i))$. Since $\psi
\in \C$, we can write $\psi(s) = \psi_{P,Q}(s) =  \int \min(p, sq)
d\lambda$ for some probability measures $P$ and $Q$ on a measurable
space $\samp$ having densities $p$ and $q$ with respect to a common
sigma finite measure $\lambda$. Without loss of generality, we assume
that  
\begin{equation}\label{tin}
  I_{f_i}(\psi) = D_{f_i}(P||Q) = D_i \qt{for $i = 1, \dots, k$}
\end{equation}
and
\begin{equation}\label{ntin}
  I_{f_i}(\psi) = D_{f_i}(P||Q) < D_i \qt{for $i = k+1, \dots, m$}. 
\end{equation}

Let $\alpha: \samp \rightarrow (-1, 1)$ be a function satisfying
\begin{equation}\label{intzeropq}
\int \alpha p d\lambda  = \int \alpha q d\lambda = 0. 
\end{equation}
Note that $(1 + \alpha)p, (1-\alpha)p, (1+\alpha)q$ and $(1 -
\alpha)q$ are all probability densities with respect to $\lambda$.
Let $P^+, P^-, Q^+,Q^-$ be probability measures having densities
$p_+:=(1+\alpha)p$, $p_-:=(1-\alpha)p$, $q_+:=(1+\alpha)q$,
$q_-:=(1-\alpha)q$ respectively with respect to $\lambda$. Also, let 
\begin{equation*}
  \psi_+(s) := \psi_{P^+, Q^+}(s) = \int (1+\alpha) \min(p, sq) d\lambda
\end{equation*}
and 
\begin{equation*}
  \psi_-(s) := \psi_{P^-, Q^-}(s) = \int (1-\alpha) \min(p, sq) d\lambda
\end{equation*}
so that $\psi = (\psi_+ + \psi_-)/2$. For every $i = 1, \dots, m$, we
observe that
\begin{align*}
  I_{f_i}(\psi_+) &= D_{f_i}(P_+||Q_+) \\
&= \int q_+ f_i \left(\frac{p_+}{q_+} \right) d\lambda + f_i'(\infty) P_+
\left\{q_+ = 0 \right\}. 
\end{align*}
Writing $(1+\alpha)p$ and $(1+\alpha)q$ for $p_+$ and $q_+$
respectively and noting that $1 + \alpha > 0$ because $\alpha$ takes
values in $(-1, 1)$, we obtain
\begin{equation}\label{plus}
  I_{f_i}(\psi_+) = I_{f_i}(\psi) + \int \alpha r_i d\lambda
\end{equation}
where
\begin{equation*}
  r_i := q f_i \left(\frac{p}{q} \right) + f_i'(\infty) p \left\{ q =
    0 \right\}. 
\end{equation*}
It follows similarly that 
\begin{equation}\label{minus}
  I_{f_i}(\psi_-) = I_{f_i}(\psi) - \int \alpha r_i d\lambda
\end{equation}
We observe that $\int r_i d\lambda \leq D_i$ for each $i = 1, \dots,
m$ which implies that 
\begin{equation}\label{j1}
\int |\alpha r_i| d\lambda < \infty
\end{equation}
for every function $\alpha$ that takes values in $(-1, 1)$ and $i = 1,
\dots, m$. 

From~\eqref{tin},~\eqref{plus} and~\eqref{minus}, it follows that the two
inequalities: 
\begin{equation}\label{intalphar}
  I_{f_i}(\psi_+) \leq D_i ~ \text{ and } ~ I_{f_i}(\psi_-) \leq D_i
\end{equation}
will be satisfied for $i = 1, \dots, k$ if and only if 
\begin{equation}\label{cond1}
  \int \alpha r_i d\lambda = 0 \qt{for $i = 1, \dots, k$}. 
\end{equation}
Moreover, from~\eqref{ntin},~\eqref{plus} and~\eqref{minus}, it
follows that if $\sup_{x \in \samp}|\alpha(x)|$ is sufficiently small,
then~\eqref{intalphar} will be satisfied also for $i = k+1, \dots,
m$. Let us say that $\alpha$ is a good function if it
satisfies~\eqref{intzeropq} and~\eqref{cond1} and if 
$\sup_{x} |\alpha(x)|$ is sufficiently small. We have thus proved that
if $\alpha$ is a good function, then both $\psi_+$ and $\psi_-$ belong
to $\cap_i \C_1(f_i, D_i)$. Because $\psi$ is extreme and $\psi =
(\psi_+ + \psi_-)/2$, we assert that $\psi = \psi_+ = \psi_-$ for
every good function $\alpha$. As a result, $\partial^l \psi(s)
= \partial^l \psi_+(s)$ for every $s > 0$ and $\partial^r \psi(s)
= \partial^r \psi_+(s)$ for every $s \geq 0$. Because of
Lemma~\ref{der} and the relations $p_+ = (1+\alpha)p$ 
and $q_+ = (1+\alpha)q$, we get that
\begin{equation}\label{g1}
  \int_{p \geq sq} \alpha q d\lambda = 0 ~~ \text{ and } ~~   \int_{p
    > sq} \alpha q d\lambda = 0 
\end{equation}
for every $s > 0$. On the other hand, the equality
$s\psi(1/s) = s \psi_+(1/s)$ for every $s > 0$ implies that $\psi_{Q, P}(s) =
\psi_{Q^+,P^+}(s)$.    Reversing the role of $q$ and $p$ in the argument that led to equation~\eqref{g1}, we equate derivatives and use $\int
\alpha p d\lambda = 0$ to get
 \begin{equation}\label{g2}
  \int_{p \geq sq} \alpha p d\lambda = 0 ~~ \text{ and } ~~   \int_{p
   > sq} \alpha p d\lambda = 0 
\end{equation}
for every $s > 0$. We have therefore shown that both~\eqref{g1}
and~\eqref{g2} hold for every $s > 0$ whenever $\alpha$ is a good
function. We now show that for every decreasing sequence $s_1 >
\dots > s_{k+2}$ of positive real numbers, the following condition
must hold  
\begin{equation}\label{cruc}
  \min_{1 \leq j \leq k+3} \left(P(B_j) + Q(B_j) \right) = 0
\end{equation}
where $B_1 = \{p \geq q s_1\}$, $B_i = \{
q s_i \leq p < q s_{i-1}\}$ for $i = 2, \dots, k+2$, and $B_{k+3} = \{p <
q s_{k+2}\}$. The proof would then be completed by Lemma~\ref{ss}. 

We prove~\eqref{cruc} via contradiction. Suppose that the
condition~\eqref{cruc} does not hold for some $s_1 > \dots >
s_{k+2}$. Let $\alpha = \sum_{j=1}^{k+3} \alpha_j I_{B_j}$ where
$\alpha_1, \dots, \alpha_{k+3}$ are real numbers in $(-1, 1)$ and
$I_{B_j}$ denotes the indicator function of $B_j$. We claim that for
this $\alpha$, the conditions~\eqref{g1} and~\eqref{g2} cannot hold
unless $\alpha_1 = \dots = \alpha_{k+3} = 0$. To see this, note
that~\eqref{g1} and~\eqref{g2} for $s = s_1$ give $\alpha_1
(P(B_1) + Q(B_1)) = 0$. But since $P(B_1) + Q(B_1)$ is
strictly positive (we are assuming that~\eqref{cruc} does not hold),
it follows that $\alpha_1=0$. We now use~\eqref{g1} and~\eqref{g2} for
$s = s_2$ to obtain $\alpha_2 = 0$. Continuing this argument, we get 
that~\eqref{g1} and~\eqref{g2} cannot hold unless $\alpha_1 = \dots =
\alpha_{k+3} = 0$. As a result, it follows that $\alpha =
\sum_{j=1}^{k+3} \alpha_j I_{B_j}$ is not a good function for every
non-zero vector $(\alpha_1, \dots, \alpha_{k+3})$ in $\R^{k+3}$.  

On the other hand, as can be easily seen by writing down the
conditions~\eqref{intzeropq} and~\eqref{cond1}, for $\alpha =
\sum_{j=1}^{k+3} \alpha_j I_{B_j}$ to be a good function, $\max_j |\alpha_j|$ needs to 
be sufficiently small and the following equalities need to be
satisfied:  
\begin{equation*}
  \sum_{j=1}^{k+3} \alpha_j P(B_j) = 0 =   \sum_{j=1}^{k+3} \alpha_j Q(B_j)
\end{equation*}
and
\begin{equation*}
  \sum_{j=1}^{k+3} \alpha_j \int_{B_j} r_i d\lambda = 0 \qt{for $i =
    1, \dots, k$}.
\end{equation*}
If~\eqref{cruc} is not satisfied, then the above represent $k+2$
linear equalities for the $k+3$ variables $\alpha_1, \dots,
\alpha_{k+3}$. Therefore, a solution exists where $\alpha_1, \dots,
\alpha_{k+3}$ are non-zero (and where $\max_j |\alpha_j|$ is small)
for which $\alpha = \sum_{j=1}^{k+3} \alpha_j I_{B_j}$ becomes a good
function. Since this is a contradiction, we have
established~\eqref{cruc}. 

By Lemma~\ref{ss}, it follows that $\psi$ can be written as $\psi_{P',
Q'}$ for two probability measures $P'$ and $Q'$ on $\{1, \dots,
k+2\}$. This proves the first part of the theorem. The case of $\cap_i
\C_2(f_i, D_i)$ is very similar. In the above argument, the only place
where we used the fact that the constraints in $\cap_i \C_1(f_i, D_i)$
are of the $\leq$ form is in asserting~\eqref{j1}. In the
case of $\cap_i \C_2(f_i, D_i)$, the statement~\eqref{j1} still holds
under the assumption that each divergence $D_{f_i}$ is finite. The
rest of the proof proceeds exactly as before. 
\end{proof}

Below, we provide the proof of Lemma~\ref{ss} which was used in the
above proof. 
\begin{proof}[Proof of Lemma~\ref{ss}]
  Let $\eta$ denote the probability measure $(P + Q)/2$. Suppose 
  \begin{equation*}
    N := \left\{x \in (0, 1): x = \eta \{p \geq q s \} \text{ for some
      } s \in (0, \infty)\right\}. 
  \end{equation*}
  We claim that $N$ is a finite set having cardinality at most
  $l-1$. To see this, suppose, if possible, that there exist points
  $0 < x_1 < \dots < x_l < 1$ in $N$. Then, we can
  write $x_i = \eta \{ p \geq q s_i\}$ for some $s_1 > \dots > s_l >
  0$. But then $\eta (B_1) = x_1$, $\eta(B_i) = x_{i} - x_{i-1} > 0$
  for $i = 2, \dots, l$ and $\eta(B_{l+1}) = 1 - x_l > 0$ which
  contradicts the condition given in the lemma. Let us therefore
  assume that the cardinality of $N$ equals $k \leq l-1$ and let $N =
  \{x_1, \dots, x_k\}$ where $0 < x_1 < \dots < x_k < 1$. Let 
  \begin{equation*}
    s_i^* := \sup \left\{s > 0 : \eta \{p \geq q s\} = x_i  \right\}
  \end{equation*}
for $i = 1, \dots, k$. Also let 
\begin{equation*}
  s_{k+1}^* := \sup \left\{s > 0: \eta\{p \geq qs \} = 1 \right\}
\end{equation*}
if there exists $s > 0$ with $\eta \{ p \geq qs \} = 1$. If there
exists no such $s > 0$, we define $s_{k+1}^* = 0$. It is easy to see
that $s_1^* \in (0, \infty]$ and $s_{k+1}^* \in [0, \infty)$ while
$s_2^*, \dots, s_k^* \in (0, \infty)$. Let us first consider the case
when $s_1^* < \infty$ and $s_{k+1}^* > 0$. In this case, for each $i =
1, \dots, k+1$, there exists a sequence $\{t_n(i)\}$ with $0 < t_n(i)
\uparrow s_{i}^*$ such that $\eta \{p \geq q t_n(i) \} = x_i$ (we take
$x_{k+1} = 1$). Because the sets $\{p \geq q t_n(i) \}$ decrease to
$\{p \geq q s_i^* \}$ as $n \rightarrow \infty$, it follows that $\eta
\{p \geq q s_i^*\} = x_i$  for each $i = 1, \dots, k+1$. Also it is easy
to see that 
\begin{equation*}
  \eta \{p > q s_i^* \}  = \lim_{s \downarrow s_i^*} \eta \{p \geq q s
  \} = x_{i-1}
\end{equation*}
for each $i = 1, \dots, k+1$ where we take $x_0 = 0$. It follows
therefore that $\eta\{p = q s_i^*\} = x_i - x_{i-1}$ for $1 \leq i
\leq k+1$. Because $\sum_{i=1}^{k+1} (x_i - x_{i-1}) = x_{k+1} - x_0 =
1$, it follows that 
\begin{equation}\label{mai}
  \sum_{i=1}^{k+1} \eta \{p = qs_i^* \} = 1. 
\end{equation}
It can be checked that the above statement is also true in the case
when $s_1^* = \infty$ and/or $s_{k+1}^* = 0$ provided we interpret 
\begin{equation*}
   \{ p = q \cdot \infty \} =  \{q = 0 \}~ \text{ and }~  \{p = q
  \cdot 0\} =  \{ p = 0 \}. 
\end{equation*}
The equality~\eqref{mai} is the same as
\begin{equation}\label{remai}
  \sum_{i=1}^{k+1} P \{p = qs_i^* \} = 1 ~ \text{ and } ~
  \sum_{i=1}^{k+1} Q \{p = q s_i^* \} = 1. 
\end{equation}
Let $p_i = P \{ p = q s_i^* \}$ and $q_i = Q\{ p = q s_i^* \}$ for
$i = 1, \dots, k+1$ so that $P' = (p_1, \dots, p_{k+1})$ and $Q' =
(q_1, \dots, q_{k+1})$ are probability measures on $\{1, \dots,
k+1\}$. For each $i = 1, \dots, k+1$, we have
\begin{equation*}
  p_i = P \{ p = q s_i^*\} = \int_{p = qs_i^*} p d\lambda = s_i^* \int_{p =
    q s_i^*} q d\lambda =  s_i^* q_i
\end{equation*}
where the above statement is to be interpreted as $q_1 = 0$ if $s_1^*
= \infty$ and as $p_{k+1} = 0$ if $s_{k+1}^* = 0$. Also
\begin{equation*}
  \int_{p = qs_i^*} \min(p, qs) d\lambda = \min(s_i^*, s) Q\{p =
  qs_i^*\} = \min(p_i, q_i s)
\end{equation*}
for every $s \geq 0$ and $i = 1, 2, \dots, k+1$. Therefore, 
\begin{align*}
  \psi_{P, Q}(s) &= \int \min(p, qs) d\lambda \\
&= \sum_{i=1}^{k+1} \int_{p = q s_i^*} \min(p, qs) d\lambda =
\psi_{P', Q'}(s). 
\end{align*}
The proof is complete because $k+1 \leq l$. 
\end{proof}
\subsection{Completion of the Proof}\label{fp}
We shall prove~\eqref{main.1}. The proof of~\eqref{main.2} is entirely
analogous. Theorem~\ref{charex} states that every function in $\cap_i
\C_1(f_i, D_i)$ that is extreme equals $\psi_{P, Q}$ for some $P, Q
\in \Ps_{m+2}$. Therefore, by Lemma~\ref{redex}, we get that $A(D_1,
\dots, D_m)$ equals 
\begin{equation*}
\sup \left\{I_f(\psi_{P, Q}): \psi_{P, Q} \in \cap_{i=1}^m \C_1(f_i,
  D_i) \text{ and } P, Q \in \Ps_{m+2} \right\}.  
\end{equation*}
Because $I_{f_i}(\psi_{P, Q})$ equals $D_{f_i}(P||Q)$, the constraint 
$\psi \in \cap_i \C_1(f_i, D_i)$ is equivalent to $D_{f_i}(P||Q) \leq
D_i$ for all $i = 1, \dots, m$. The proof is therefore complete. 

\section{Remarks and Extensions}\label{remex}
\subsection{Stronger Version}\label{stv}
The proof of Theorem~\ref{main} actually yields a smaller expression
for $A(D_1, \dots, D_m)$ than $A_{m+2}(D_1, \dots, D_m)$ and a larger
expression for $B(D_1, \dots, D_m)$ than $B_{m+2}(D_1, \dots,
D_m)$. For each subset $J$ of $\{1, \dots, m\}$, let $A^J(D_1, \dots,
D_m)$ denote the supremum of $D_f(P||Q)$ over all probability measures
$P, Q \in \Ps_{k+2}$ (where $k$ is the cardinality of $J$) for which
$D_{f_i}(P||Q) = D_i$ for $i \in J$ and $D_{f_i}(P||Q) < D_i$ for $i
\notin J$. It is clear that  
\begin{equation*}
  A^J(D_1, \dots, D_m) \leq A_{m+2}(D_1, \dots, D_m)
\end{equation*}
for each $J \subseteq \{1, \dots, m\}$. The following is therefore a
stronger version of Theorem~\ref{main}: 
\begin{equation}\label{sva}
  A(D_1, \dots, D_m) = \max_{J \subseteq \{1, \dots, m \}} A^J(D_1,
  \dots, D_m) 
\end{equation}
An analogous statement also holds for $B(D_1, \dots, D_m)$. Let us now
show that our proof of Theorem~\ref{main} given in Section~\ref{fp}
results in~\eqref{sva}. By Theorem~\ref{charex}, every function $\psi$ in
$\cap_i \C_1(f_i, D_i)$ that is extreme equals $\psi_{P, Q}$ for some
$P, Q \in \Ps_{k+2}$ where $k$ is the number of indices $i$ for which
$I_{f_i}(\psi) = D_{f_i}(P||Q) = D_i$. Therefore, if $J$ denotes these
indices, then 
\begin{align*}
  I_f(\psi) = D_f(P||Q) &\leq A^J(D_1, \dots, D_m) \\
 &\leq \max_{J \subseteq \{1, \dots, m \}}
  A^J(D_1, \dots, D_m)
\end{align*}
for every $\psi \in ext(\cap_i \C_1(f_i, D_i))$. The
equality~\eqref{sva} therefore follows from Lemma~\ref{redex}.  

\subsection{Joint Ranges}\label{hash}
Recall that the joint range of divergences $D_{f_1}, \dots, D_{f_m}$
is denoted by $\jr(f_1, \dots, f_m)$ and is defined as the set of all
vectors in $\R^m$ that equal $(D_{f_1}(P||Q), \dots, D_{f_m}(P||Q))$
for some $P$ and $Q$. The quantities $A(D_1, \dots, D_m)$ and $B(D_1,
\dots, D_m)$ can easily be calculated from knowledge of $\jr(f, f_1,
\dots, f_m)$. It therefore makes sense to try to prove
Theorem~\ref{main} by trying to determine the joint range $\jr(f, f_1,
\dots, f_m)$. We argue here that this approach is not good enough to
prove Theorem~\ref{main}; it results in the weaker
identities~\eqref{ref1} and~\eqref{ref2}. 

In the following theorem, we characterize the joint range $\jr(f_1,
\dots, f_m)$ for every arbitrary set of $m$ divergences. We show that
it suffices to restrict attention to pairs of probability measures in
$\Ps_{m+2}$. For each $k \geq 1$, let 
\begin{equation*}
  \jr_k(f_1, \dots, f_m) := \left\{(D_{f_1}(P||Q), \dots,
    D_{f_m}(P||Q)) : P, Q \in \Ps_k \right\}. 
\end{equation*}
\begin{theorem}\label{hvgen}
  For every $m \geq 1$ and divergences $D_{f_1}, \dots, D_{f_m}$, we
  have 
  \begin{equation*}
    \jr(f_1, \dots, f_m) = \jr_{m+2}(f_1, \dots, f_m). 
  \end{equation*}
\end{theorem}
For the special case $m = 2$, this theorem has
already been proved by~\cite{HarremoesVajda}. The short proof given
below uses the Caratheodory theorem and was communicated to us by an
anonymous referee. In contrast, the proof given
in~\cite{HarremoesVajda} for $m = 2$ is much more elaborate. The
counterexamples in~\cite{HarremoesVajda} show the tightness of this
theorem. After the proof, we describe an attempt to prove
Theorem~\ref{main} via Theorem~\ref{hvgen}. 
 \begin{proof}
  We just need to prove that $\jr(f_1, \dots, f_m) \subseteq
  \jr_{m+2}(f_1, \dots, f_m)$. Let $u \in \jr(f_1, \dots, f_m)$. Then
  $u = (D_{f_1}(P||Q), \dots, D_{f_m}(P||Q))$ for some pair of
  probability measures $P$ and $Q$. If $p$ and $q$ denote the
  densities of $P$ and $Q$ with respect to a common measure $\lambda$,
  then  
  \begin{equation}\label{nz1}
 u = \int_{\{ q > 0\}} \left(f_1 \left(\frac{p}{q} \right), \dots, f_m
 \left(\frac{p}{q} \right)\right) dQ +  P 
    \{q = 0 \} \left(f_1'(\infty), \dots,
 f_m'(\infty) \right). 
  \end{equation}
  Let $S \subseteq \R^{m+1}$ be defined by $S := \{(s, f_1(s), \dots,
  f_m(s)): s \geq 0 \}$. Then clearly the vector  
  \begin{equation*}
    \int_{\{ q > 0\}} \left(\frac{p}{q}, f_1 \left(\frac{p}{q}
      \right), \dots, f_m \left(\frac{p}{q} \right)\right) dQ 
  \end{equation*}
  lies in the convex hull of $S$. Because $S$ is a connected subset of
  $\R^{m+1}$, we can use Caratheodory's theorem (see, for
  example,~\cite{BaranyKarasev}) to assert that any point in its
  convex hull can be   written as a convex combination of at most
  $m+1$ points in $S$. As a result, we can write 
  \begin{equation}\label{nz2}
    \int_{\{ q > 0\}} \left(\frac{p}{q}, f_1 \left(\frac{p}{q}
      \right), \dots, f_m \left(\frac{p}{q} \right)\right) dQ  =
    \sum_{i=1}^{m+1} \alpha_i \left(s_i, f_1(s_i), \dots, f_m(s_i)
    \right) 
  \end{equation}
  for some $\alpha_1, \dots, \alpha_{m+1} \geq 0$ with $\sum_{i}
  \alpha_i = 1$ and $s_1, \dots, s_{m+1} \geq 0$. One consequence of
  this representation is that 
  \begin{equation}\label{ru}
    \sum_{i=1}^{m+1} \alpha_i s_i = \int_{q > 0}\left( \frac{p}{q}
    \right) dQ = P \{q > 0 \}. 
  \end{equation}
  We now define two probability measures $P'$ and $Q'$ in $\Ps_{m+2}$
  as follows: $P'\{i+1\} = \alpha_i s_i$ for $1 \leq i \leq m+1$ and
  $P'\{1\} = P\{q=0\}$; and $Q'\{i+1\} = \alpha_i$ for $1 \leq i \leq
  m+1$ and $Q'\{1\} = 0$. The fact that $\sum_{i=1}^{m+2} P'\{i\} =
  1$ follows from~\eqref{ru}. The equalities~\eqref{nz1}
  and~\eqref{nz2} together clearly imply that $u = (D_{f_1}(P'||Q'),
  \dots, D_{f_m}(P'||Q'))$. Thus $u \in \jr_{m+2}(f_1, \dots, f_m)$
  and this completes the proof. 
\end{proof}
Clearly $A(D_1, \dots, D_m)$ and $B(D_1, \dots, D_m)$ can be written
as functions of the joint range $\jr(f, f_1, \dots,
f_m)$. Theorem~\ref{hvgen} immediately therefore implies
\begin{equation}\label{ref1}
A(D_1, \dots, D_m) = A_{m+3}(D_1, \dots, D_m)  
\end{equation}
and 
\begin{equation}\label{ref2}
  B(D_1, \dots, D_m) = B_{m+3}(D_1, \dots, D_m). 
\end{equation}
These results are clearly weaker than those given by
Theorem~\ref{main}. Strictly speaking, one can deduce a slightly
stronger conclusion than~\eqref{ref1} and~\eqref{ref2} from
Theorem~\ref{hvgen}. A probability measure on $\{1, \dots, m+3\}$ is
determined by $m+2$ real numbers. Therefore, a pair of probability
measures in $\Ps_{m+3}$ are determined by $2m + 4$ real numbers. The
inequalities~\eqref{ref1} and~\eqref{ref2} therefore reduce the
optimization problems for $A(D_1, \dots, D_m)$ and $B(D_1, \dots,
D_m)$ into optimization problems over $2m + 4$ variables. A closer
inspection at the proof of Theorem~\ref{hvgen} shows that one actually
gets a reduction to $2m + 3$ variables. This is because the
probability measure $Q'$ in the proof satisfies $Q'\{1\} =
0$. Therefore, by an argument based solely on the joint 
range of $D_f, D_{f_1}, \dots, D_{f_m}$, one can
reduce the optimization problems for $A(D_1, \dots, D_m)$ and $B(D_1,
\dots, D_m)$ into optimization problems over $2m + 3$
variables. Because of the tightness of Theorem~\ref{hvgen}, this is
the best reduction that one can hope for the quantities $A(D_1, \dots,
D_m)$ and $B(D_1, \dots, D_m)$ via an argument based on the joint
range alone. On the other hand, Theorem~\ref{main} achieves a
reduction to $2m+2$ variables. 

\subsection{Tightness}\label{tightness}
The conclusion of Theorem~\ref{main} is tight in the sense
that, in general, one cannot reduce the optimization problems to pairs
of probability measures on spaces of cardinality strictly smaller than
$m+2$. We shall demonstrate this fact in this section by means of an
example. We also explain this fact numerically in
Example~\ref{numtight}.   

Consider the problem of maximizing an $f$-divergence subject to
a upper bound on the total variation distance. In other words, let
\begin{equation*}
  A(V) := \sup\{D_f(P||Q) : V(P, Q) \leq V \}
\end{equation*}
where $D_f$ is an arbitrary $f$-divergence. In this case,
Theorem~\ref{main} asserts that $A(V)$ equals $A_3(V)$ where, as before,  
\begin{equation*}
  A_k(V) := \sup \{ D_f(P||Q): P, Q \in \Ps_k, V(P, Q) \leq V \}. 
\end{equation*}
We shall show below that when $D_f$ is a finite divergence and when
$f$ is strictly convex on $(0, \infty)$, the quantity $A_3(V)$ is
strictly larger than $A_2(V)$ for all $V \in (0, 1)$.  

The quantity $A_3(V) = A(V)$ can be determined precisely. The easiest
way is to use Lemma~\ref{lv}. Because  
\begin{equation*}
  V(P, Q) = D_{u_1}(P||Q) = 1 - \psi_{P, Q}(1), 
\end{equation*}
the constraint $V(P, Q) \leq V$ is equivalent to $\psi_{P, Q}(1) \geq
1 - V$. Therefore, by Lemma~\ref{lv}, we get 
\begin{equation*}
  A(V) = \sup \{I_f(\psi) : \psi \in \C \text{ and } \psi(1) \geq 1-V
  \}.  
\end{equation*}
It is obvious that the supremum above is achieved for $\psi(s) = (1-V)
\min(1, s)$ which equals $\psi_{P', Q'}$ for $P' = (1-V, V, 0)$ and
$Q' = (1-V, 0, V)$. Thus
\begin{equation*}
  A(V) = D_f(P'||Q') = V \left(f(0) + f'(\infty) \right). 
\end{equation*}
In other words, by Remark~\ref{remfin}, the quantity $A(V)$ equals $V$
times the maximum possible value of the divergence $D_f$. 

Let us now consider the quantity $A_2(V)$. By compactness and the form
of the constraint, it follows that there exist two probability
measures $P^*$ and $Q^*$ in $\Ps_2$
with $V(P^*, Q^*) = V$ and $D_f(P^*||Q^*) = A_2(V)$. We can then,
without loss of generality, parametrize $P^*$ and $Q^*$ by $P^* =
(\rho, 1-\rho)$ and $Q^* = (\rho + V, 1 - \rho - V)$ for some $0 \leq
\rho \leq 1-V$. Consider now the probability measures 
\begin{equation*}
  \tilde{P} = \left(\frac{\rho}{2}, \frac{\rho}{2}, 1 - \rho \right)
   \text{  and  }  \tilde{Q} = \left(\frac{\rho}{2} +
    \frac{V}{4}, \frac{\rho}{2} + \frac{3V}{4}, 1 - \rho - V \right)
\end{equation*}
in $\Ps_3$. If $V \in (0, 1)$, by strict convexity of the function
$f$, it is easy to see that 
\begin{equation*}
  D_f(\tilde{P}||\tilde{Q}) > D_f(P^*||Q^*) = A_2(V). 
\end{equation*}
On the other hand, it is easy to see that $V(\tilde{P}, \tilde{Q})$
equals $V$ and hence $A_3(V) > D_f(\tilde{P}||\tilde{Q})$. Therefore,
$A_3(V) > A_2(V)$. Thus, Theorem~\ref{main} is tight in
general. However, in some special cases, one can obtain stronger
conclusions, see Sections~\ref{secprim} and~\ref{secchi}. 

\subsection{Finiteness assumption for $B(D_1, \dots, D_m)$}
In order to prove~\eqref{main.2}, we required that all the
divergences $D_{f_1}, \dots, D_{f_m}$ are finite. The reason is mainly
technical and the finiteness assumption was crucially used in the
proof of Lemma~\ref{redex}. The set $\cap_i \C_2(f_i, D_i)$ will not
be closed (in $C[0, \infty)$ equipped with the metric $\rho$) if some
of the divergences $D_{f_i}$ were non-finite (closedness of $\cap_i
\C_2(f_i, D_i)$ was critical in the application of Choquet's theorem in
Lemma~\ref{redex}). To illustrate this non-closedness, let us consider
$m = 1$ and the set $\C_2(f_1, D_1)$ for some non-finite divergence
$D_{f_1}$ and $D_1 > 0$. By~\eqref{finchar},  because $D_{f_1}$ is
non-finite, we have 
\begin{equation*}
  \int_0^{\infty} \min(1, s) d\nu_{f_1}(s) = \infty. 
\end{equation*}
The function $\psi_0(s) = \min(1, s)$ clearly does not belong to
$\C_2(f_1, D_1)$ because $I_{f_1}(\psi_0) = 0$. But we shall show that
$\psi_0$ belongs to the closure of $\C_2(f_1, D_1)$. Indeed, if
\begin{equation*}
  \psi_n(s) := \left(1 - \frac{1}{n} \right) \min(1, s) \qt{for $s
  \geq 0$}, 
\end{equation*}
then clearly $\psi_n$ converges to $\psi$ in the metric
$\rho$. Moreover, for each $n$, $\psi_n \in \C$ and 
\begin{equation*}
  I_{f_1}(\psi_n) = \frac{1}{n}   \int_0^{\infty} \min(1, s)
  d\nu_{f_1}(s) = \infty.  
\end{equation*}
Thus $\psi_n \in \C_2(f_1, D_1)$ for each $n \geq 1$ which implies
that $\psi_0$ belongs to the closure of $\C_2(f_1, D_1)$. Therefore,
$\C_2(f_1, D_1)$ is not closed. 

The quantity $B(D_1, \dots, D_m)$ behaves strangely when some of the
divergences $D_{f_i}$ are non-finite and when $D_f$ is finite. Indeed,
in this case, one can simply drop the constraints corresponding to the
non-finite divergences and reduce the problem to the case when all
divergences are finite. This is the content of the next lemma. 
\begin{lemma}\label{yu}
  Let $D_f, D_{f_1}, \dots, D_{f_m}$ be finite divergences and let
  $D_{f_{m+1}}, \dots, D_{f_{m+l}}$ be non-finite divergences. Then 
  \begin{equation*}
    B(D_1, \dots, D_{m+l}) = B(D_1, \dots, D_m)
  \end{equation*}
\end{lemma}
\begin{proof}
  We shall work with~\eqref{s1b}. Because $\cap_{i=1}^{m+l} \C_2(f_i,
  D_i)$ is contained in $\cap_{i=1}^{m} \C_2(f_i, D_i)$, it follows
  that $B(D_1, \dots, D_{m+l})$ is larger than or equal to $B(D_1,
  \dots, D_m)$. To prove the other inequality, let $\psi \in
  \cap_{i=1}^{m} \C_2(f_i, D_i)$. For each $n \geq 1$, define
  \begin{equation*}
    \psi_n(s) = \min \left[\left( 1- \frac{1}{n} \right) \min(1, s),
      \psi(s) \right]  
  \end{equation*}
  It is easy to check that $\psi_n \in \C$. Note that for $1 \leq i
  \leq m$, 
  \begin{align*}
    I_{f_i}(\psi_n) &= \int_0^{\infty} \left(\min(1, s) - \psi_n(s)
    \right) d\nu_{f_i}(s) \\
&\geq \int_0^{\infty} \left(\min(1, s) - \psi(s)
    \right) d\nu_{f_i}(s) = I_{f_i}(\psi) \geq D_i. 
  \end{align*}
  Moreover, for $m < i \leq m+l$, we have
  \begin{align*}
    I_{f_i}(\psi_n) &= \int_0^{\infty} \left(\min(1, s) - \psi_n(s)
    \right) d\nu_{f_i}(s) \\
&\geq \frac{1}{n}\int_0^{\infty} \min(1, s)  d\nu_{f_i}(s) =
\infty \geq D_i.  
  \end{align*}
It therefore follows that $\psi_n \in \cap_{i=1}^{m+l} \C_2(f_i, D_i)$
for every $n \geq 1$. Consequently, 
\begin{equation*}
  I_{f}(\psi_n) \geq B(D_1, \dots, D_{m+l}) \qt{for every $n \geq
    1$}. 
\end{equation*}
Observe that $\psi_n(s)$ converges to $\psi(s)$ for every $s \geq
0$. Thus, because $D_f$ is a finite divergence, it follows by the
dominated convergence theorem that $I_f(\psi_n)$ converges to
$I_f(\psi)$ which results in 
\begin{equation*}
  I_f(\psi) \geq B(D_1, \dots, D_{m+l}). 
\end{equation*}
Finally, because $\psi \in \cap_{i=1}^m \C_2(f_i, D_i)$ is arbitrary,
we have proved that $B(D_1, \dots, D_m)$ is larger than or equal to
$B(D_1, \dots, D_{m+l})$ which completes the proof of the lemma. 
\end{proof}
\begin{remark}
  If $D_f$ is finite and if all the divergences $D_{f_1}, \dots,
  D_{f_m}$ are non-finite, then Lemma~\ref{yu} gives that 
  \begin{equation}\label{yu.ken}
B(D_1, \dots,   D_m) = 0    
  \end{equation}
 for all values of $D_1, \dots, D_m$. Here is a special instance of
 this result. Suppose that $D_f$ denotes the total variation 
 distance, $m = 1$ and that $D_{f_1}$ is the Kullback-Leibler
 divergence. Then~\eqref{yu.ken} shows that the smallest value of the
 total variation distance over all probability measures with
 Kullback-Leibler divergence at least 5 (say) equals 0. The same
 conclusion holds for multiple non-finite divergence constraints as
 well.  
\end{remark}

Theorem~\ref{main} gives a formula for $B(D_1, \dots, D_m)$ for
arbitrary $D_f$ and for finite $D_{f_1}, \dots, D_{f_m}$. In
Lemma~\ref{yu}, we showed that when $D_f$ is finite, then the case
when one of more of $D_{f_1}, \dots, D_{f_m}$ are non-finite can be
reduced to the case where all the constraint divergences are finite
which is handled by Theorem~\ref{main}. The case that we are unable to
resolve is $B(D_1, \dots, D_m)$ when $D_f$ is
non-finite and when one or more of $D_{f_1}, \dots, D_{f_m}$ are
non-finite. This case is neither covered by Theorem~\ref{main} nor by
Lemma~\ref{yu}.  

\subsection{Sufficiency of the extreme point characterization}
In Theorem~\ref{charex}, we gave a necessary condition for functions
in the classes $\cap_i \C_1(f_i, D_i)$ and $\cap_i \C_2(f_i, D_i)$ to
be 
extreme. As we have seen, this necessary condition was enough to prove
Theorem~\ref{main}. For the sake of completeness, in this section, we
investigate whether the condition in Theorem~\ref{charex} is
sufficient as well for extremity.  

Let $j \in \{1, 2\}$ and let $\psi$ be a function in $\cap_i \C_j(f_i,
D_i)$. Suppose $\psi$ satisfies the condition given in
Theorem~\ref{charex} i.e., let $\psi = \psi_{P, Q}$ for two
probability measures $P, Q \in \Ps_{k+2}$ where $k$ is the number of
indices where $I_{f_i}(\psi) = D_i$. Here, we explore the question of
extremity of $\psi$ in $\cap_i \C_j(f_i, D_i)$. 

Let $l \leq k+2$ be the size of the (finite) support set of the
measure $P+Q$ and let $P = \{p_1, \dots, p_{l}\}$ and $Q = \{q_1,
\dots, q_{l}\}$, then $\psi(s) = \sum_{i=1}^{l} \min \left(p_i, q_i s
\right)$. Because the size of the support set of $P+Q$ is $l$, it
follows that $\max(p_i, q_i) > 0$ for every $i$. It is easy to check
that $\psi$ is piecewise linear with knots at $p_i/q_i$ (this ratio
can equal 0 or $\infty$ as well). 

Suppose that $\psi = (\psi_1 + \psi_2)/2$ for two functions $\psi_1$
and $\psi_2$ in $\cap_i \C_j(f_i, D_i)$. Because $\psi_1$ and $\psi_2$
are both concave, it follows that they both have to be linear in the
regions where $\psi$ is linear. As a result, one can write
\begin{equation*}
  \psi_1(s) = \sum_{i=1}^l (1+\alpha_i) \min(p_i, q_i s)
\end{equation*}
and 
\begin{equation*}
\psi_2(s) = \sum_{i=1}^l (1 - \alpha_i) \min(p_i, q_i s)
\end{equation*}
for some $\alpha_1, \dots, \alpha_n \in [-1, 1]$ satisfying
\begin{equation}\label{mee}
\sum_{i=1}^l \alpha_i p_i =  \sum_{i=1}^l \alpha_i q_i = 0. 
\end{equation}
Now, whenever $I_{f_i}(\psi) = D_i$, because of the above, we must
have $I_{f_i}(\psi_1) = D_i$. This latter equality can be written as a
linear equality in $\alpha_1, \dots, \alpha_l$. Because $I_{f_i}(\psi)
= D_i$ for $k$ indices $i$, we obtain $k$ linear equations for
$\alpha_1, \dots, \alpha_l$. These, together with~\eqref{mee}, give
rise to $k+2$ linear equations for the $l \leq k+2$ variables
$\alpha_1, \dots, \alpha_l$. Under appropriate linear independence
conditions on the measures $\nu_{f_i}$, these would imply that
$\alpha_i = 0$ for every $1 \leq i \leq l$ which would further imply
that $\psi_1 = \psi = \psi_2$ and that $\psi$ is extreme. 

In the case when $m \leq 1$ however, no such explict linear
independence conditions are necessary and, moreover, one can also give
a geometric proof of the sufficiency characterization of the extreme
points. We do this below in two parts: Lemma~\ref{mos} deals with $m
= 0$ (i.e., extreme points of $\C$) and Lemma~\ref{m1s} deals with
the $m = 1$ case. 
\begin{lemma}\label{mos}
 For every $P, Q \in \Ps_2$, the function $\psi_{P, Q}$ is extreme in
 $\C$. 
\end{lemma}
\begin{proof}
  Fix two probability measures $P$ and $Q$ on $\{1, 2\}$ and let $J$
  denote the smallest open interval (possibly infinite) such
  that $\psi_{P, Q}(s) = \min(1, s)$ for $s \notin J$. By explicitly
  writing down the expression for $\psi$ in terms of $P\{1\}$ and
  $Q\{1\}$, it is easy to see that if $J$ is non-empty, then
  $\psi_{P, Q}$ is linear on $J$.  

  Suppose now that $\psi_{P, Q}$ equals the convex combination
  $(\psi_1 + \psi_2)/2$ for two functions $\psi_1$ and $\psi_2$ in
  $\C$. If $J$ is empty, then $\psi_{P, Q}$ equals the function
  $\min(1, s)$ for all $s$ and since all functions in $\C$ and bounded  
  from above by $\min(1, s)$, it follows that
  \begin{equation}\label{kd}
\psi_{P, Q}(s) = \psi_1(s) =   \psi_2(s) = \min(1, s)     
  \end{equation}
  for all $s \geq 0$. 

  Let us therefore assume that $J$ is non-empty. In this case,
  again it is obvious that~\eqref{kd} holds for $s \notin
  J$. Concavity of functions in $\C$ and linearity of $\psi$ in
  $J$ would then imply that $\psi_1 \geq \psi_{P, Q}$ and $\psi_2 \geq
  \psi_{P, Q}$. Since $\psi_{P, Q}$ is the average of 
  $\psi_1$ and $\psi_2$, this can happen only when $\psi_{P, Q} =
  \psi_1 = \psi_2$. The proof is complete. 
\end{proof}
\begin{lemma}\label{m1s}
  Let $j \in \{1, 2\}$ and consider the class $\C_j(f_1, D_1)$ for
  $D_1 > 0$. For
  every $P, Q \in \Ps_3$ with $D_{f_1}(P|| Q) = D_1$, the function
  $\psi_{P, Q}$ is  extreme in $\C_j(f_1, D_1)$. 
\end{lemma}
\begin{proof}
  Fix two probability measures $P$ and $Q$ in $\Ps_3$ with
  $D_{f_1}(P||Q) = D_1$ so that $I_{f_1}(\psi_{P, Q}) = D_1$. For
  notational convenience, let us denote $\psi_{P, Q}$ by $\psi$. As in 
  the proof of Lemma~\ref{mos}, let $J$ denote the smallest interval
  outside which $\psi(s)$ equals $\min(1, s)$. If $J$ is empty,
  then $\psi$ equals the function $\min(1, s)$ which is
  obviously extreme. So let us assume that $J$ is non-empty. In that
  case, because $P, Q \in \Ps_3$, it can be checked that $\psi$
  is piecewise linear with at most two segments in $J$. 

  Suppose that $\psi = (\psi_1 + \psi_2)/2$ for two functions
  $\psi_1, \psi_2 \in \C_j(f_1, D_1)$. Because, $I_{f_1}(\psi)
  = D_{f_1}(P||Q) = D_1$, it follows that
  \begin{equation}\label{mrk}
   I_{f_1}(\psi_1) = I_{f_1}(\psi_2) = I_{f_1}(\psi) =
   D_{f_1}(P||Q) = D_1.  
  \end{equation}
  If $\psi$ has exactly one segment in $J$, then, by concavity,
  the inequalities $\psi_1(s) \geq \psi(s)$ and $\psi_2(s) \geq
  \psi(s)$ hold for all $s$. Because $\psi_1$ and $\psi_2$
  average out to $\psi$, we must then have $\psi = \psi_1 = \psi_2$.   

  Now suppose that $\psi$ has exactly two segments in
  $J_{\psi}$. Let $a$ be the point in $J$ such that $\psi$ is linear
  on both $J \cap [0, a]$ and $J \cap [a, \infty)$. We 
  shall show that $\psi(a) = \psi_1(a) = \psi_2(a)$. Concavity of
  $\psi_1$ and $\psi_2$ and linearity of $\psi$ on $J \cap [0, a]$ and
  $J \cap [a, \infty)$ can then be used to show that $\psi = \psi_1 =
  \psi_2$. Suppose, if possible, that $\psi_1(a) > \psi(a)$. Using the
  concavity of $\psi_1$, it then follows that $\psi_1(s) > \psi(s)$
  for all $s \in J$. Because of~\eqref{mrk}, it follows that 
  \begin{equation*}
    \int_J \left(\psi_1(s) - \psi(s) \right) d\nu_{f_1}(s) =
    \int_0^{\infty} \left(\psi_1(s) - \psi(s) \right)d\nu_{f_1}(s) = 0 
  \end{equation*}
  This implies that $\nu_{f_1}(J) = 0$. But then 
  \begin{equation*}
    D_1 = I_{f_1}(\psi) = \int_J \left(\min(1, s) - \psi(s) \right)
    d\nu_{f_1}(s) = 0
  \end{equation*}
  which contradicts the fact that $D_1 > 0$. We have thus obtained
   $\psi_1(a) \leq \psi(a)$. Similarly, \makebox{$\psi_2(a) \leq \psi(a)$}
  and since $\psi(a)$ is an average of $\psi_1(a)$ and $\psi_2(a)$, it
  follows that $\psi(a) = \psi_1(a) = \psi_2(a)$. The proof is
  complete. 
\end{proof}

\section{Applications and Special Cases}\label{apsp}
\subsection{Primitive Divergences}\label{secprim}
In this section, we consider the case of the quantity $B(D_1,
\dots, D_m)$ where all the divergences $D_{f_1}, \dots, D_{f_m}$ are
primitive divergences (see Remark~\ref{prim}). In Theorem~\ref{myprim}
below, we show that, in this case, $B(D_1, \dots, D_m)$ actually
equals $B_{m+1}(D_1, \dots, D_m)$ as opposed to $B_{m+2}(D_1, \dots,
D_m)$. 

The problem of minimizing an $f$-divergence subject to constraints on
primitive divergences and the related problem of obtaining
inequalities between $f$-divergences and primitive divergences has
received much attention in the literature and has a long history. Let
us briefly mention some important works in this area. The most
well-known such inequality is Pinsker's  inequality which states that
$D_{KL}(P||Q) \geq 2 V^2(P, Q)$ where $D_{KL}$ is the Kullback-Leibler
divergence which corresponds to $f(x) = x \log x$ and $V$ is the total
variation distance. Pinsker~\cite{PinskerIneq} proved this inequality
with the constant 2 replaced by 1. The inequality with the constant 2
(which cannot be improved further) has been proved independently
almost at the same time by Csiszar~\cite{Csiszar66},
Kemperman~\cite{kemperman69} and Kullback~\cite{Kullback67IEEE}.  

Although Pinsker's inequality is very useful, it is not sharp
in the sense that 
\begin{equation*}
  \inf \left\{ D_{KL}(P||Q) : V(P, Q) \geq V \right\} > 2 V^2
\end{equation*}
for every $V \neq 0$. The problem of finding sharp inequalities
between $D_{KL}(P||Q)$ and $V(P, Q)$ was solved
in~\cite{RefinePinsker} where an implicit expression for the
infimum in the left hand side above was provided. 

The more general problem of finding the best lower bound for an
arbitrary $f$-divergence given a lower bound on total variation
distance was solved by Gilardoni in~\cite{Gilardoni06}. The problem of
finding lower bounds for $f$-divergences given constraints on a finite
number of primitive divergences was studied by~\cite{GenPin}. In
Remark~\ref{rwbash}, we explain how our theorem below gives an
equivalent but simpler solution compared to the solution
of~\cite{GenPin}. 
\begin{theorem}\label{myprim}
  Suppose that $D_f$ is an arbitrary divergence and that all
  divergences $D_{f_1}, \dots, D_{f_m}$ are primitive
  divergences. Then 
  \begin{equation*}
    B(D_1, \dots, D_m) = B_{m+1}(D_1, \dots, D_m). 
  \end{equation*}
\end{theorem}
\begin{proof}
  Theorem~\ref{main} states that $B(D_1, \dots, D_m)$ equals
  $B_{m+2}(D_1, \dots, D_m)$. We shall show therefore that
  $B_{m+2}(D_1, \dots, D_m)$ equals $B_{m+1}(D_1, \dots, D_m)$. 

  It is obvious that 
  \begin{equation*}
B_{m+2}(D_1, \dots, D_m) \leq B_{m+1}(D_1,
  \dots, D_m)     
  \end{equation*}
because we have a minimization problem and the constraint set is
larger in the case of $B_{m+2}(D_1, \dots, D_m)$. It is therefore
enough to prove that 
\begin{equation*}
B_{m+2}(D_1, \dots, D_m) \geq B_{m+1}(D_1, \dots, D_m).  
\end{equation*}
Fix two probability measures $P = (p_1, \dots, p_{m+2})$ and $Q =
(q_1, \dots, q_{m+2})$ in $\Ps_{m+2}$ with $D_{f_i}(P||Q) \geq D_i$
for every $i = 1, \dots, m$. We show below that
\begin{equation*}
D_f(P||Q) \geq B_{m+1}(D_1, \dots, D_m)  
\end{equation*}
which will complete the proof.   

Without loss of generality, we assume that $p_i + q_i > 0$ for each
$i$ and that the likelihood ratios $r_i := p_i/q_i \in [0, \infty]$
satisfy $r_1 \leq \dots \leq r_{m+2}$. Because each divergence
$D_{f_i}$ is assumed to be primitive, the convex function $f_i$ is
piecewise linear with exactly two linear parts. As a result, there
exists some index $j \in \{1, \dots, m+1\}$ such that all the
functions $f_1, \dots, f_m$ are linear in the interval $[r_j,
r_{j+1}]$. 

Now consider the two probability measures $P^*$ and $Q^*$ in
$\Ps_{m+1}$ defined by 
\begin{equation*}
  P^* := (p_1, \dots, p_{j-1}, p_j+p_{j+1}, p_{j+2}, \dots, p_{m+2})
\end{equation*}
and 
\begin{equation*}
  Q^* := (q_1, \dots, q_{j-1}, q_j+q_{j+1}, q_{j+2}, \dots, q_{m+2})
\end{equation*}
Because of the linearity of $f_1, \dots, f_m$ on $[r_j, r_{j+1}]$, it
is easy to check that 
\begin{equation*}
  D_{f_i}(P^*||Q^*) = D_{f_i}(P||Q) \geq D_i \qt{for all $1 \leq i
    \leq m$}. 
\end{equation*}
As a result, we have 
\begin{equation*}
 D_f(P^*||Q^*) \geq B_{m+1}(D_1, \dots, D_m). 
\end{equation*}
On the other hand, by convexity or as a consequence of the data
processing inequality for $f$-divergences (see, for
example,~\cite[Lemma 4.1]{CsiszarShields}), it follows that 
\begin{equation*}
  D_f(P||Q) \geq D_{f}(P^*||Q^*) \geq B_{m+1}(D_1, \dots, D_m). 
\end{equation*}
The proof is complete. 
\end{proof}

\begin{remark}\label{rwbash}
Let $0 < s_1 < \dots < s_m < \infty$ and let $D_{f_i}$ be the
primitive divergence corresponding to $f_i = u_{s_i}$ (the
functions $u_{s_i}$ are defined in Remark~\ref{prim}). Then the
optimization problem corresponding to $B_{m+1}(D_1, \dots, D_m)$ can
be written as:  
\begin{equation}\label{rwb}
\begin{aligned}
& \underset{p,q \in [0,1]^{m+1}}{\textnormal{minimize}}
& & \sum_{j: q_j > 0} q_j f\left(\frac {p_j} {q_j} \right) +f'(\infty)
\sum_{j:q_j = 0} p_j \\ 
& \textnormal{subject to}
& & p_j\ge0, \; q_j\ge 0 \textnormal{ for all } j = 1, \dots, m+1 \\
& & & \sum p_j  = \sum q_j = 1\\
& & & \sum_{j} \min \left(p_j, q_j s_i \right)  \leq \min(1, s_i) - D_i
\end{aligned}
\end{equation}
for $i = 1, \dots, m$. According to Theorem~\ref{myprim}, the optimal
value of this problem equals $B(D_1, \dots, D_m)$. As we mentioned
before, the problem of determining $B(D_1, \dots, D_m)$ when the
divergences $D_{f_i}$ are all primitive divergences has been studied
by~\cite{GenPin}. Their main result~\cite[Theorem 6]{GenPin} gives a
characterization of $B(D_1, \dots, D_m)$ that is much more complicated
than~\eqref{rwb}. However, the two forms are essentially
equivalent. To understand the equivalence, observe that, by
Lemma~\ref{lv}, $D_f(P||Q)$ can be written as an integral functional
of $\psi_{P, Q}$. It is possible to precisely characterize the form of
the function $\psi_{P, Q}$ when $P, Q \in \Ps_{m+1}$. As a result, 
the optimization problem~\eqref{rwb} can be reformulated in
terms of such concave functions $\psi$. This, after some tedious
algebra, leads to the formula for $B(D_1, \dots, D_m)$ given
in~\cite[Theorem 6]{GenPin}. Our formula~\eqref{rwb} is much simpler
and, moreover, is conceptually easier to understand. 
\end{remark}

The special case of $m = 1$ in Theorem~\ref{myprim} asserts that in
order to determine $B(D)$ when $D_{f_1}$ is a primitive 
divergence, one only needs to consider probabilities on $\{1,
2\}$. This fact is well-known at least in the case when $D_{f_1}$ is
the total variation distance (see, for example,~\cite[Proposition
2.1]{Gilardoni06}). It is then possible to give a more direct
expression for $B(D)$ which is the content of the following lemma, whose
special case for $s = 1$ appears in~\cite[Proposition
2.1]{Gilardoni06}.   

\begin{lemma}
  Let $m = 1$ and consider the quantity $B(D)$ where $D_f$
  is an arbitrary $f$-divergence and $D_{f_1}$ is the primitive
  divergence corresponding to $f_1 = u_s$ for a fixed $s > 0$. Then,
  for every $0 \leq D \leq \min(1, s)$, the quantity $B(D)$ equals
  \begin{equation}\label{tus}
\inf_{0 \leq q \leq H/s} \left[ (1-q) f \left(\frac{H - qs}{1-q} \right) + q f
      \left(\frac{1 + qs - H}{q} \right) \right]
  \end{equation}
where $H := \min(1, s) - D$. 
\end{lemma}
\begin{proof}
  We shall now show that $B_2(D)$ equals~\eqref{tus}. Note that
  $B_2(0) = 0$ and~\eqref{tus} also equals 0 when $D = 0$. To see
  this, note that it is trivially zero (because $f(1) = 0$) when $s =
  1$ and when $s \neq 1$, then it is zero because the value at $q =
  (1-\min(1, s))/(1-s)$ equals 0. So we shall assume below that $D >
  0$. The optimization problem corresponding to $B_2(D)$ is: 
\begin{equation}\label{ks}
\begin{aligned}
& \underset{p,q \in [0,1]^{2}}{\textnormal{minimize}}
& & \sum_{j: q_j > 0} q_j f\left(\frac {p_j} {q_j} \right) +f'(\infty)
\sum_{j:q_j = 0} p_j \\ 
& \textnormal{subject to}
& & p_j\ge0, \; q_j\ge 0 \textnormal{ for } j = 1, 2 \\
& & &  p_1 + p_2  = q_1 + q_2 = 1\\
& & & \min(p_1, q_1 s) + \min(p_2, q_2 s) = H.
\end{aligned}
\end{equation}
Note that we have equality as opposed to $\leq$ in the last constraint
above. This is because of the fact that for every $(p_1, p_2)$ and
$(q_1, q_2)$ lying in the constraint set for which the last constraint
is not tight, we can get $(\tilde{p}_1, \tilde{p}_2)$ and
$(\tilde{q}_1, \tilde{q}_2)$ still lying in the constraint set with
the last constraint satisfied with an equality sign and for which the
objective function is reduced. 

We will now finish the proof by showing that the optimal value of the
optimization problem~\eqref{ks} is~\eqref{tus}. Let $(p_1, p_2)$ and
$(q_1, q_2)$ satisfy the constraint set with $p_1/q_1 \leq 1 \leq
p_2/q_2$. If $s \notin [p_1/q_1, p_2/q_2]$, then clearly $\min(p_1,
q_1 s) + \min(p_2, q_2 s) = \min(1, s)$ and such $(p_1, p_2)$ and
$(q_1, q_2)$ do not satisfy the constraint set because $D > 0$. So we
assume that $s \in [p_1/q_1, p_2/q_2]$. In this case, the final
constraint gives $p_1 = H - q_2 s$. We can therefore write each of
$p_1, p_2$ and $q_1$ in terms of $q_2$. Plugging these values in the
objective function leads to the function in~\eqref{tus} (with $q$
replaced by $q_2$). The fact that each of $p_1, p_2, q_1$ and $q_2$
need to lie between 0 and 1 gives the constraint $0 \leq q_2 \leq
H/s$. The proof is complete.   
\end{proof}

For completeness, let us note the special case of the above lemma in
the case of the total variation distance, which corresponds to $s =
1$. This result is due to Gilardoni~\cite[Proposition
2.1]{Gilardoni06}. 

\begin{corollary}[Gilardoni]\label{frd}
  Let $m = 1$ and consider the quantity $B(V)$ where $D_f$ is an
  arbitrary $f$-divergence and $D_{f_1}(P||Q)$ equals $V(P, Q)$, the
  total variation distance between $P$ and $Q$. Then, for every $0
  \leq V \leq 1$, 
  \begin{equation}\label{frd.1}
    B(V) := \inf \left\{T(q, V) : 0 \leq q \leq 1-V \right\}
  \end{equation}
where
\begin{equation*}
T(q, V) := (1-q) f \left(\frac{1-V-q}{1-q} \right) + q f \left(\frac{q+V}{q}
\right). 
\end{equation*}
Consequently, for every pair of probability  measures $P$ and $Q$, we
have the inequality
 \begin{equation}\label{frd.2}
D_f(P||Q) \geq \inf \left\{T(q, V(P, Q)) : 0 \leq q \leq 1 - V(P, Q)
\right\} 
 \end{equation}
Moreover, this represents the sharpest possible inequality between
$D_f$ and total variation distance.  
\end{corollary}

Although the expression~\eqref{frd.1} cannot be simplified further in
general, one can get much simpler expressions for $B(V)$ in certain
special cases. One such special case of interest corresponds to
\textit{symmetric} $f$-divergences. An $f$-divergence is said to be
symmetric if the underlying convex function $f$ satisfies the
identity: $f(x) = x f(1/x)$ for all $x \in (0, \infty)$. It is easy to
check that under this condition, one has $D_f(P||Q) = D_f(Q||P)$ for
all $P$ and $Q$. Examples of symmetric divergences include the total
variation distance, squared Hellinger distance, triangular
discrimination and the Jensen-Shannon divergence. The following result
is due to Gilardoni~\cite{Gilardoni06}. We include it here for
completeness and also because our proof is more direct than that 
in~\cite{Gilardoni06}. 
\begin{corollary}[Gilardoni]\label{symmrd}
  Let $m = 1$ and consider the quantity $B(V)$ where $D_f$ is a
  symmetric $f$-divergence and $D_{f_1}(P||Q)$ equals $V(P, Q)$, the
  total variation distance between $P$ and $Q$. Then, for every $0
  \leq V \leq 1$, 
  \begin{equation}\label{symmrd.1}
B(V) = (1-V) f \left(\frac{1+V}{1-V} \right). 
  \end{equation}
 Consequently, for every pair of probability
 measures $P$ and $Q$, we have 
 \begin{equation}\label{symmrd.2}
   D_f(P||Q) \geq (1-V(P, Q)) f \left(\frac{1+V(P, Q)}{1-V(P, Q)} \right). 
 \end{equation}
Moreover, this represents the sharpest possible inequality between the
symmetric divergence $D_f$ and total variation distance.  
\end{corollary}
\begin{proof}
  We shall show that the right hand side of~\eqref{frd.1} equals the
  right hand side of~\eqref{symmrd.1} when $D_f$ is a symmetric
  divergence. Consider the quantity $T(q, V)$ defined in
  Corollary~\ref{frd}. Because $f(x) = xf(1/x)$, it can be easily
  checked that 
  \begin{equation*}
    T(q, V) = T(1-q-V, V) \qt{for all $q \in [0, 1-V]$}. 
  \end{equation*}
In other words, the function $q \mapsto T(q, V)$ is symmetric in the
interval $[0, 1-V]$ about the mid-point $(1-V)/2$. Moreover, as can be
checked by taking derivatives (one-sided derivatives if $f$ is not
differentiable), $q \mapsto T(q, V)$ is convex on $[0, 1-V]$ (this fact
does not require $f$ to be symmetric). These two facts clearly imply
that 
\begin{equation*}
 \inf_{0 \leq q \leq 1-V} T(q, V) = T \left(\frac{1-V}{2}, V \right)
  = (1-V) f \left(\frac{1+V}{1-V} \right)
\end{equation*}
which completes the proof. 
\end{proof}

\subsection{Chi-squared divergence}\label{secchi}
In this section, we describe another situation where the conclusion of
Theorem~\ref{main} can be further simplified. 
\begin{theorem}
  Let $m=1$ and consider the quantity $A(D)$ where $D_f$ is the
  chi-squared divergence, $\chi^2(P||Q)$ which corresponds to $f(x) :=
  x^2 - 1$. Also let the function $f_1$ be such that the function
  $h:(0, \infty) \rightarrow (0, \infty)$ defined by $h(x) :=
  (1+f_1(x))/x$ is a strictly increasing, strictly convex, twice
  differentiable bijective mapping. Then $A(D) = h^{-1}(D+1) - 1$,
  where $h^{-1}$ denotes the inverse function of $h$ on $(0, \infty)$. 
\end{theorem}
\begin{proof}
By Theorem~\ref{main}, $A(D)$ equals the optimal value of the problem: 
\begin{equation*}
\begin{aligned}
& \underset{p,q \in [0,1]^{3}}{\textnormal{maximize}}
& & \sum_{j: q_j > 0} \frac{p_j^2}{q_j} - 1 +\infty \cdot
\sum_{j:q_j = 0} p_j \\ 
& \textnormal{subject to}
& & p_j \ge 0, \; q_j \ge 0 \textnormal{ for all } j = 1, 2, 3\\
& & & \sum p_j  = \sum q_j = 1\\
& & & \sum_{j: q_j > 0} q_j f_1 \left(\frac {p_j} {q_j} \right) +
f_1'(\infty) \sum_{j:q_j = 0} p_j \le D
\end{aligned}
\end{equation*}
By convexity of $h$, we have
\begin{equation}\label{kuku}
  h(x) \geq h(a) + h'(a)(x-a) 
\end{equation}
for every $x > 0$ and $a > 0$. One consequence of this and the fact
that $h$ is strictly increasing is that
\begin{equation*}
  h(1) + h'(1) (x-1) \leq  h(x) \leq h(1)
\end{equation*}
for all $x \in (0, 1)$. This implies that $\lim_{x \downarrow 0} x
h(x) = 0$ and as a result 
\begin{equation*}
  f_1(0) = \lim_{x \downarrow 0} f_1(x) = \lim_{x \downarrow 0}
  \left(x h(x) - 1 \right) = -1
\end{equation*}
Further, because $h$ is strictly increasing, we have $h'(a) > 0$ and
thus   
\begin{equation*}
  f_1'(\infty) = \lim_{x \rightarrow \infty} h(x) = \infty
\end{equation*}
which implies that we only need to consider $P$ and $Q$ for which
$\sum_{j: q_j = 0} p_j = 0$. Writing~\eqref{kuku} in terms of
$f_1(x)$, we obtain 
\begin{equation*}
  1+f_1(x) \geq x \left(h(a) - ah'(a) \right) + x^2 h'(a). 
\end{equation*}
for every $x > 0$ and also at $x = 0$ (because $f_1(0) := \lim_{x
  \downarrow 0} f_1(x)$). Applying this inequality to $x = p_j/q_j$ for
$q_j > 0$ and then multiplying by $q_j$, we obtain 
\begin{equation*}
  q_j + q_j f_1(p_j/q_j)  \geq p_j \left(h(a) - ah'(a) \right) +
  \frac{p_j^2}{q_j} h'(a)
\end{equation*}
for each $j = 1, 2, 3$. As a result, we get
\begin{equation*}
h'(a) \sum_{j:q_j > 0} \frac{p_j^2}{q_j} \leq \sum_{j:q_j > 0} q_j f_1
\left(\frac{p_j}{q_j} \right) + 1 - h(a) + ah'(a)
\end{equation*}
Because $P$ and $Q$ satisfy the constraint, we have 
\begin{equation*}
  \sum_{j:q_j > 0} q_j f_1 \left(\frac{p_j}{q_j} \right) \leq D
\end{equation*}
and hence
\begin{equation*}
\sum_{j:q_j > 0} \frac{p_j^2}{q_j} - 1 \leq  \left[\frac{
    D + 1 - h(a) + ah'(a)}{h'(a)} \right] - 1. 
\end{equation*}
Because $a > 0$ is arbitrary, we get 
\begin{equation*}
  A(D) \leq \inf_{a > 0}\left[\frac{
    D + 1 - h(a) + ah'(a)}{h'(a)} \right] - 1. 
\end{equation*}
By elementary algebra, the above infimum is achieved at $a^* =
h^{-1}(D+1)$ and we then obtain $A(D) \leq h^{-1}(D+1) - 1$. To see
that $A(D)$ is exactly equal to $h^{-1}(D+1) - 1$, observe that the
probabilities $P = (1, 0, 0)$ and $Q = (1/a^*, 1-1/a^*, 0)$ satisy
$D_{f_1}(P||Q) = D$ and $\chi^2(P||Q) = h^{-1}(D+1) - 1$. 
\end{proof}
The function $f_1(x) = x^l - 1$ for $l > 2$ clearly satisfies the
conditions of the above theorem. We therefore obtain the following
result as a simple corollary. 
\begin{corollary}
  Let $m = 1$ and consider the quantity $A(D)$ where $D_f(P||Q) =
  \chi^2(P||Q)$ and $D_{f_1}$ is the power divergence,
  $D^{(l)}(P||Q)$, corresponding to $f_1(x) = x^l - 1$ for $l >
  2$. Then $A(D) = (1+D)^{1/(l-1)} - 1$. This yields the sharp
  inequality
  \begin{equation*}
    \chi^2(P||Q) + 1 \leq \left(1+D^{(l)}(P||Q) \right)^{1/(l-1)}
  \end{equation*}
  between the chi-squared divergence and power divergence for $l >
  2$. 
\end{corollary}

\section{Numerical Computation}\label{nucom}
In this section we explore numerical methods for solving the
optimization problems~\eqref{finiteA} and~\eqref{finiteB} in order to
compute $A(D_1,\dots,D_m)$ and $B(D_1,\dots,D_m)$ respectively. In
Section~\ref{leca}, we consider the special case when $D_f$ is a
primitive divergence. This special case is motivated by the
statistical problem of obtaining lower bounds for the minimax risk and
we show that the quantity $A(D_1, \dots, D_m)$ can be computed
exactly via convex optimization for every $m \geq 1$ and every
arbitrary choice of $D_{f_1}, \dots, D_{f_m}$. In Section~\ref{heer},
we consider the special case $m = 1$ and demonstrate
that~\eqref{finiteA} and~\eqref{finiteB} can be solved for practically
any pair of $f$-divergences by a gridded search over the
low-dimensional parameter speace. We verify several known inequalities
and also improve on some existing inequalities that are not sharp. 

\subsection{Maximizing Primitive Divergences}\label{leca}
In this subsection we consider maximizing a primitive divergence
subject to upper bounds on arbitrary $f$-divergences.  While this
optimization problem is not a-priori convex, we reduce it to a
collection of convex problems. 

The optimization problem~\eqref{finiteA} where $D_f$ is a primitive
divergence is, of course, closely related to the problem of bounding
from above a primitive divergence subject to upper bounds on other
$f$-divergences. This latter problem arises in obtaining lower bounds
for the minimax risk in nonparametric statistical
estimation (see, for example,~\cite{GuntuFdiv, GuntuThesis, Yu97lecam, 
  Tsybakovbook}). For example, Le Cam's inequality, which is a 
popular technique for obtaining minimax lower bounds, says that the
minimax risk is bounded from below by a multiple of the $L_1$ affinity
between two probability measures $P$ and $Q$, where the $L_1$ affinity
is defined as $1-V(P,Q)$.  The $L_1$ affinity also appears in
Assouad's Lemma, another technique for  obtaining minimax lower
bounds.  Evaluating $V(P,Q)$ is hard because 
$P$ and $Q$ are typically product distributions of the form $P =
\otimes_{i=1}^n P_i$ (or mixtures of such distributions), so it is
difficult to express $V(P,Q)$ in terms of $V(P_i,Q_i)$ (which can be
easier to compute). 

Application of Le Cam's inequality in practice, therefore, requires
one to obtain a good upper bound on the  total variation, $V(P,Q)$.   
One typically first bounds $D_{f}(P||Q)$ for an $f$-divergence that
decouples for product distributions such as squared Hellinger,
chi-squared, or Kullback-Leibler divergence and then translates this
into a bound on $V(P,Q)$. It is common to use crude bounds like
Pinsker's inequality for this purpose and we believe there is room for
improvement by using tight bounds. Also, one typically uses only 
one $f$-divergence to bound $V(P, Q)$; but we shall argue here
that one gets better bounds (Figure~\ref{improvement}) when using
multiple divergences simultaneously. This is one of our motivations
for studying the case $m \geq 2$ as opposed to just $m = 1$. The
constants underlying minimax lower bounds might be improved by the use
of these better bounds addressing a common criticism  of minimax lower
bound techniques.  

Theorem~\ref{exact} below solves the problem of maximizing a primitive
divergence $D_{u_s}$ given constraints on $m$ other divergences
$D_{f_i}$ exactly via convex optimization.  This leads to a fast
algorithm with well-studied convergence properties. 

For each $m \geq 1$, let
\begin{equation*}
\mathcal{S}_m = \{\sigma \in \{-1,1\}^{m+2} : \sigma_i \leq \sigma_j
\textnormal{ for } i \leq j\}   
\end{equation*}
For each $\sigma \in \mathcal{S}_m$, let us consider the following
\textit{convex} optimization problem and denote its optimal value by
$V_{\sigma}(D_1, \dots, D_m)$. 
\begin{equation}\label{conopt}
\begin{aligned}
& \underset{p,q \in [0,1]^{m+2}}{\textnormal{maximize}}
& & \sum_{j=1}^{m+2} \sigma_j \left(p_j - s q_j \right)  \\ 
& \textnormal{subject to}
& & p_j \ge 0, \; q_j \ge 0 \textnormal{ for all } j = 1, \dots, m+2\\
& & & \sum p_j  = \sum q_j = 1\\
& & & \sum_{j: q_j > 0} q_j f_i \left(\frac {p_j} {q_j} \right) + f_i '
(\infty) \sum_{j:q_j = 0} p_j \le D_i
\end{aligned}
\end{equation}
for $i = 1, \dots, m$. Note that this problem is convex because the
objective function is linear and the constraint set is convex in $p_1,
\dots, p_{m+2}, q_1, \dots, q_{m+2}$. The fact that the constraint set
is convex is a consequence of the convexity of $D_{f_i}(P||Q)$
in $(P, Q)$ (see, for example,~\cite[Lemma 4.1]{CsiszarShields}). It
is also clear that this is a $2m+2$-dimensional optimization problem
because there are $2m+4$ variables in all which satisfy two linear
equality constraints. 

\begin{theorem}\label{exact}
Let $D_f$ denote the primitive $f$-divergence corresponding to $f =
u_s$ for some $s > 0$. Then 
\begin{equation}\label{convexA}
\begin{aligned}
A(D_1,\dots,D_m) = - \frac { \left | s - 1 \right | } 2 + \max_{\sigma
  \in \mathcal{S}_m} V_{\sigma}(D_1, \dots, D_m)
\end{aligned}
\end{equation}
Consequently, $A(D_1, \dots, D_m)$ can be computed by solving the
$|\mathcal{S}_m| = m+3$ convex optimization problems~\eqref{conopt}. 
\end{theorem}
\begin{proof}
Theorem \ref{main} asserts that $A(D_1, \dots, D_m)$ equals the
optimal value of the optimization problem~\eqref{finiteA}. Note that
the constraint sets of the problems~\eqref{finiteA} and~\eqref{conopt}
are the same. Let us denote this constraint set by $\bbC_m$ so that 
$$A(D_1,\dots,D_m) = \max_{P,Q\in \bbC_m} D_{u_s}(P||Q).$$  
The objective of~\eqref{finiteA} can be written as
\begin{equation*}
\begin{aligned}
D_{u_s} (P || Q) &= \min(1,s) - \sum_{j=1}^{m+2} \min(p_j,sq_j) \\
& = \min(1,s) - \frac 1 2 \sum_{j=1}^{m+2} p_j + sq_j - \left | p_j - sq_j \right | \\
& = - \frac { \left | s - 1 \right | } 2 + \max_{\sigma \in \{ -1,
  1\}^{m+2} }  \sum_{j=1}^{m+2} \sigma_j \left ( p_j - sq_j\right ). 
\end{aligned}
\end{equation*}
Because two maxima can always be interchanged, we have 
\begin{equation*}
\begin{aligned}
\underset{P,Q \in \bbC_m}{\max}  \left [\max_{\sigma \in
    \{-1,1\}^{m+2}}\sum_{j=1}^{m+2} \sigma_j(p_j-sq_j) \right ]  
 =  \max_{\sigma \in \{-1,1\}^{m+2}} \left [\max_{P,Q \in
     \mathbb{C}_m} \sum_{j=1}^{m+2}\sigma_j(p_j-sq_j)\right ].  
\end{aligned}
\end{equation*}
Note that the inner maximization in the right hand side above is precisely
the convex problem~\eqref{conopt}.

Because the optimal value of \eqref{conopt} is invariant to
permuting the indices of $\sigma$, we have the
reduction 
$$\max_{\sigma \in \{-1,1\}^{m+2}} \max_{P,Q \in \mathbb{C}_m}
\sigma^T(P-sQ) = \max_{\sigma \in \mathcal{S}_m} \max_{P,Q \in
  \mathbb{C}_m} \sigma^T(P-sQ).$$ 
This shows that we can restrict attention only to those
problems~\eqref{conopt} for $\sigma \in \mathcal{S}_m$. It is obvious
that $|\mathcal{S}_m| = m+3$. The proof is complete. 
\end{proof}

\begin{example}\label{hellinger}
Consider the special case of Theorem~\ref{exact} when $m = 1$, $s = 1$
and when $D_{f_1}$ is the squared Hellinger distance which corresponds
to $f_1(x) = (\sqrt{x} - 1)^2/2$. In other words, we consider the
problem of maximizing the total variation distance subject to an upper
bound on the Hellinger distance. The solution to this problem given by
Theorem~\ref{exact} is plotted in Figure~\ref{fig:simple}(a). Each red
dot shows $A(H) =: A_H^{TV}(H)$ computed by solving the four
4-dimensional convex optimization problems~\eqref{conopt} (each
corresponding to a $\sigma \in \mathcal{S}_1$). 

Note that the quantity $A_H^{TV}(H)$ can be obtained analytically in
a closed form. Indeed, since $f_1$ is a symmetric divergence, the sharp
inequality bounding the total variation distance by the squared
Hellinger distance is given by~\eqref{symmrd.2} with $f(x) =
(\sqrt{x}-1)^2$ (this inequality is usually attributed
to~\cite{LeCam:86book}) which implies that   
\begin{equation*}
  A^{TV}_H(H) = \sqrt{2H} \sqrt{1 - \frac{H}{2}}.
\end{equation*}
We have plotted this function analytically by the solid cyan line in
Figure~\ref{fig:simple}(a). It is clear that our numerical
optimization method given by Theorem~\ref{exact} agrees with the known
analytical bound. 
\end{example}

\begin{example}\label{kl}
For another simple application of Theorem~\ref{exact}, consider
maximizing the total variation subject to an upper bound on the
Kullback-Leibler divergence. In other words, we take $m = 1$, $s = 1$
and $f_1(x) = x\log x$ and plot the solution given by
Theorem~\ref{exact} in Figure~\ref{fig:simple}(b). Each black dot
shows $A(K) =: A^{TV}_{KL}(K)$ for a different value of $K$,
computed by solving the four 4-dimensional convex optimization
problems~\eqref{conopt}. The solid green line shows Pinsker's analytic
upper bound $\sqrt{2 K}$ which is not sharp for any $K>0$. 
\end{example}

\begin{figure}
	\centering
	\begin{subfigure}[b]{\textwidth}     
               \centering
              \quad \includegraphics[width=4.5in]{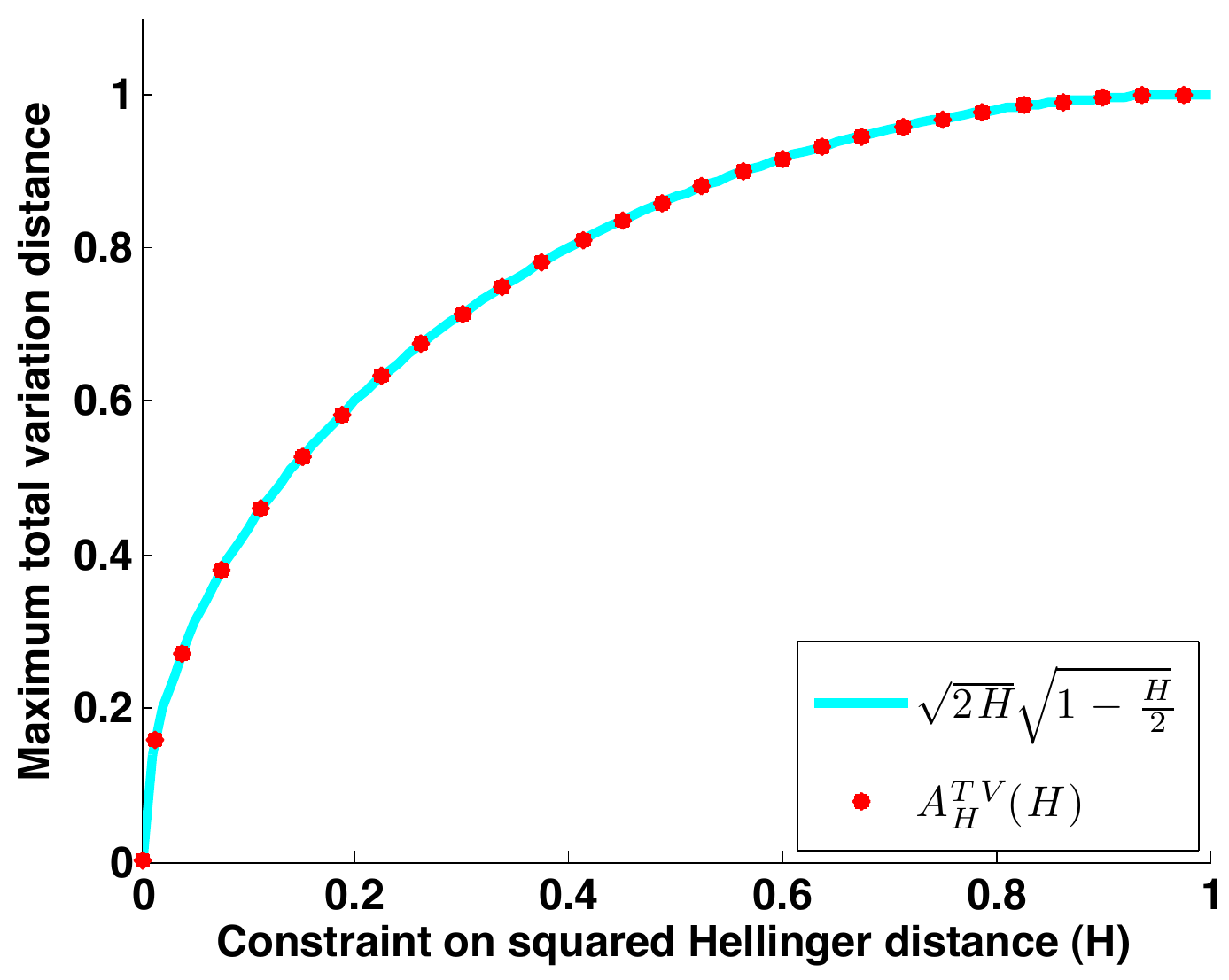}               
	\end{subfigure}\\
	\begin{subfigure}[b]{\textwidth}
		\centering
		\quad\includegraphics[width=4.5in]{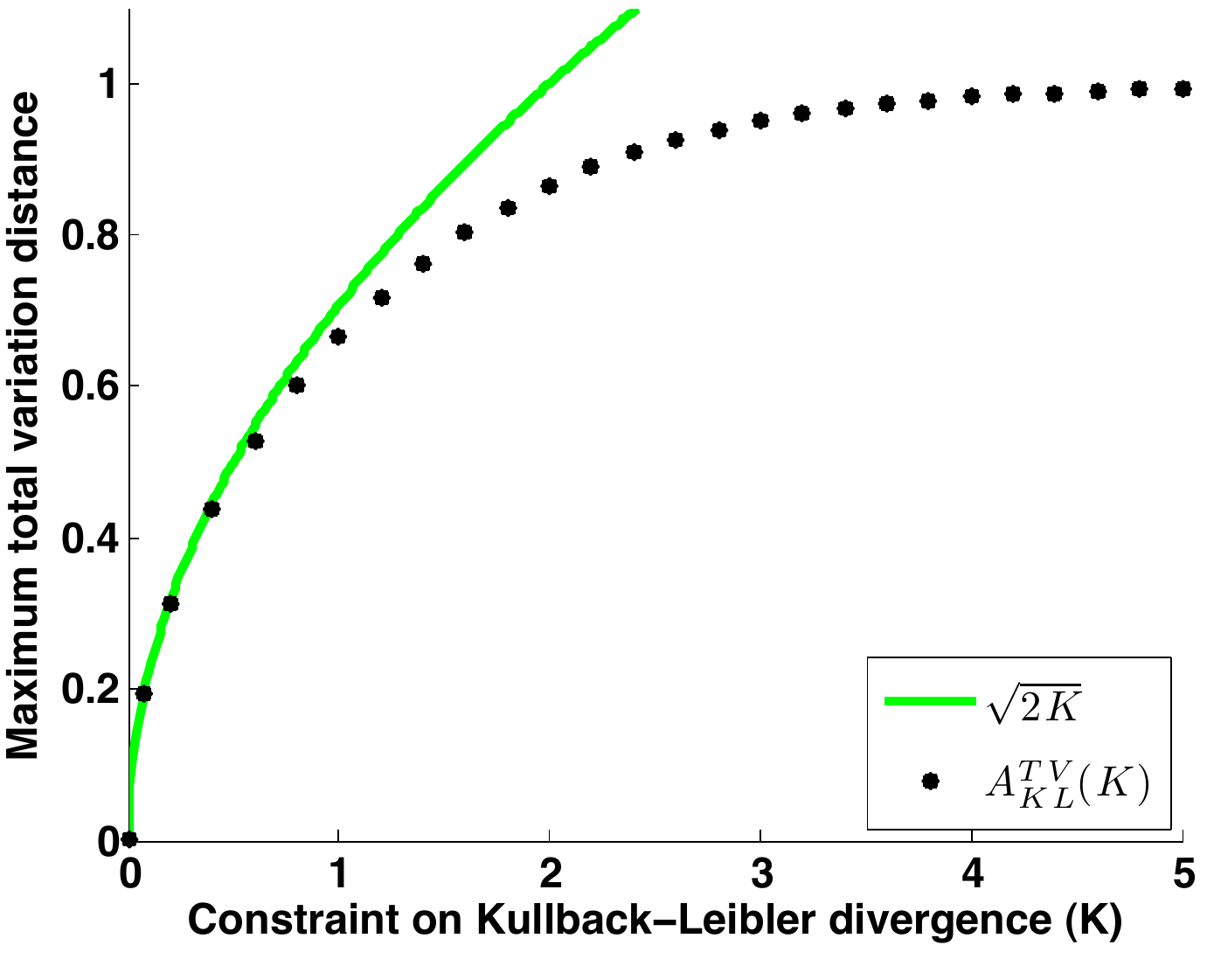}
	\end{subfigure}
        \caption{Two simple applications of Theorem~\ref{exact} discussed in
          examples \ref{hellinger} and \ref{kl}.  Here and in all subsequent plots we set the axis limits to the maximum value of the relevant $f$-divergence and to 5 in the case of the Kullback-Leibler divergence (which has no maximum value).} 
        \label{fig:simple}
\end{figure}

\begin{example}\label{tvhkl}
We now consider maximizing the total variation subject to constraints
on both the Hellinger distance and Kullback-Leibler divergence. In
other words, we take $m = 2$, $s = 1$, $f_1(x) = (\sqrt{x} - 1)^2/2$
and $f_2(x) = x \log x$. To the best of our knowledge, there does not
exist a closed form analytical solution to this problem. However,
numerical solution is straightforward by Theorem~\ref{exact} as shown
below. 

According to Theorem~\ref{exact}, for fixed $H,K \ge 0$ we can
compute $A(H,K) =: A^{TV}_{HKL}(H,K)$ by solving five
6-dimensional convex programs~\eqref{conopt}. Figure
\ref{twoConstraints} shows the function $A^{TV}_{HKL}(H,K)$ interpolated from 14884 
$(H,K)$ pairs.  We used CVX in MATLAB to solve the convex
programs.  The height of each point in the surface shows how large the
total variation can be when the squared Hellinger distance and
Kullback-Leibler divergence are bounded by $H$ and $K$
respectively.  As expected, the total variation is zero when either
$H = 0$ or $K = 0$, and it approaches 1 for large values of
$H$ and $K$.  Next, observe that the surface
$A_{HKL}^{TV}(H,K)$ is flat as $K$ varies for small $H$, and
vice-versa flat as $H$ varies for small $K$.  This is because only
one constraint is tight in these regions.  In other words, the surface
$A^{TV}_{HKL}(H,K)$ is approximately the point-wise minimum of the
two surfaces $A^{TV}_H(H)$ and $A^{TV}_{KL}(K)$, with a diagonal
ridge at the intersection of these two surfaces.  But, as can be seen
in Figure \ref{improvement}, our bound that simultaneously leverages
both single-coordinate bounds is strictly better than the simple
minimum of those two individual bounds for some $(H,K)$. In other
words, there exist $(H, K)$ such that 
\begin{equation}\label{inform}
\min \left(A^{TV}_H(H), A^{TV}_{KL}(K) \right) - 
A^{TV}_{HKL}(H,K) > 0
\end{equation}
The left hand side above is positive when both single-coordinate
bounds are informative, i.e. when both constraints in the optimization
problem~\eqref{finiteA} are active.  We will explain later (see Example
\ref{ex:hkl} and Figure \ref{fig:hkl}) that the location of
this ridge is predicted by an inequality between $D_H(P||Q)$ and
$D_{KL}(P||Q)$. 
\end{example}

\begin{figure}[h]
\centering
\includegraphics[width=4.5 in]{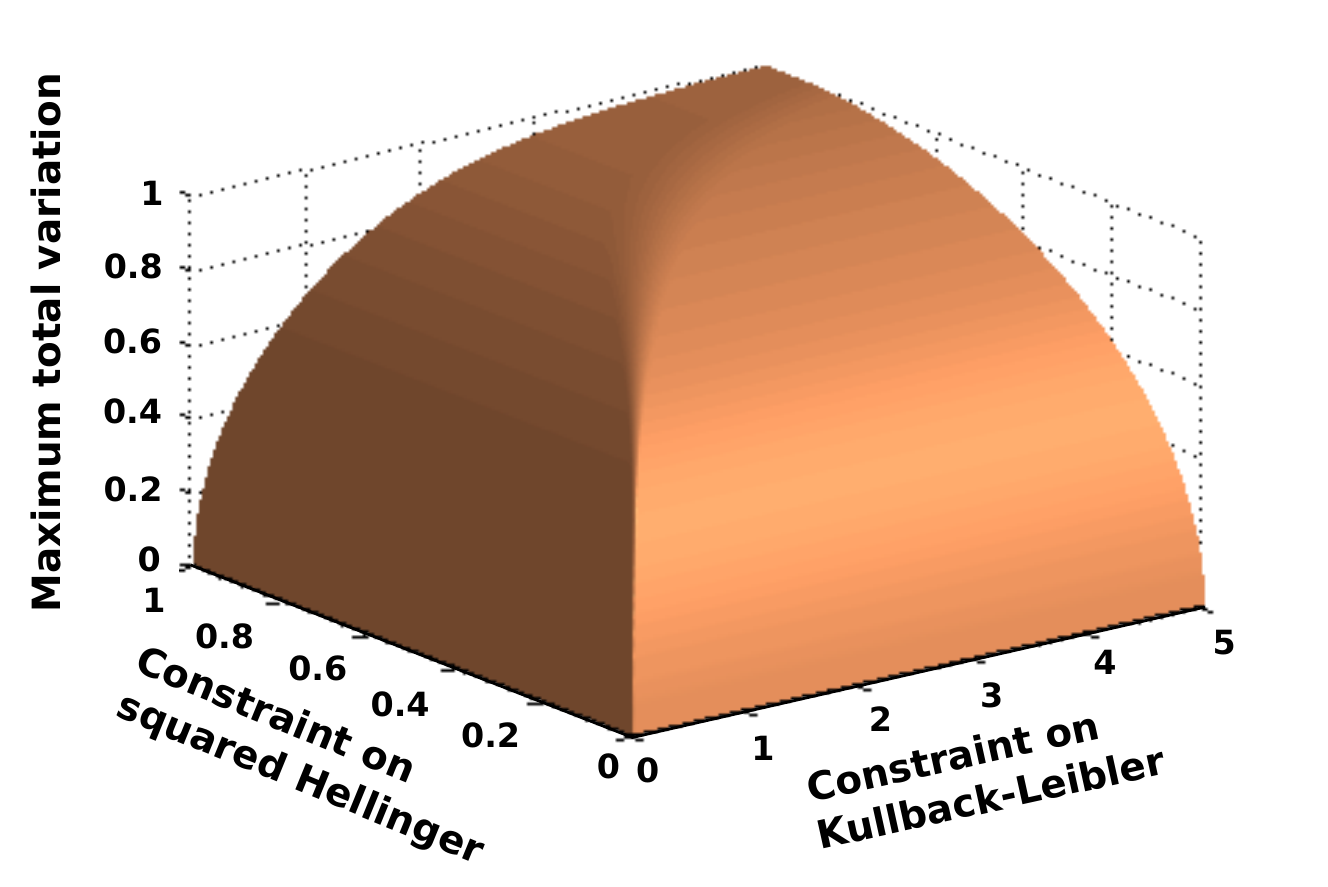}
\caption{The height of each point in the surface above shows
  $A^{TV}_{HKL}(H,K)$ for a different $(H,K)$ pair--the the
  least upper bound on total variation when squared Hellinger distance
  and Kullback-Leibler divergence are bounded by $H$ and $K$
  respectively (see example \ref{tvhkl}).} 
\label{twoConstraints}
\end{figure}

\begin{figure}[h]
\centering
\includegraphics[width=4.5 in]{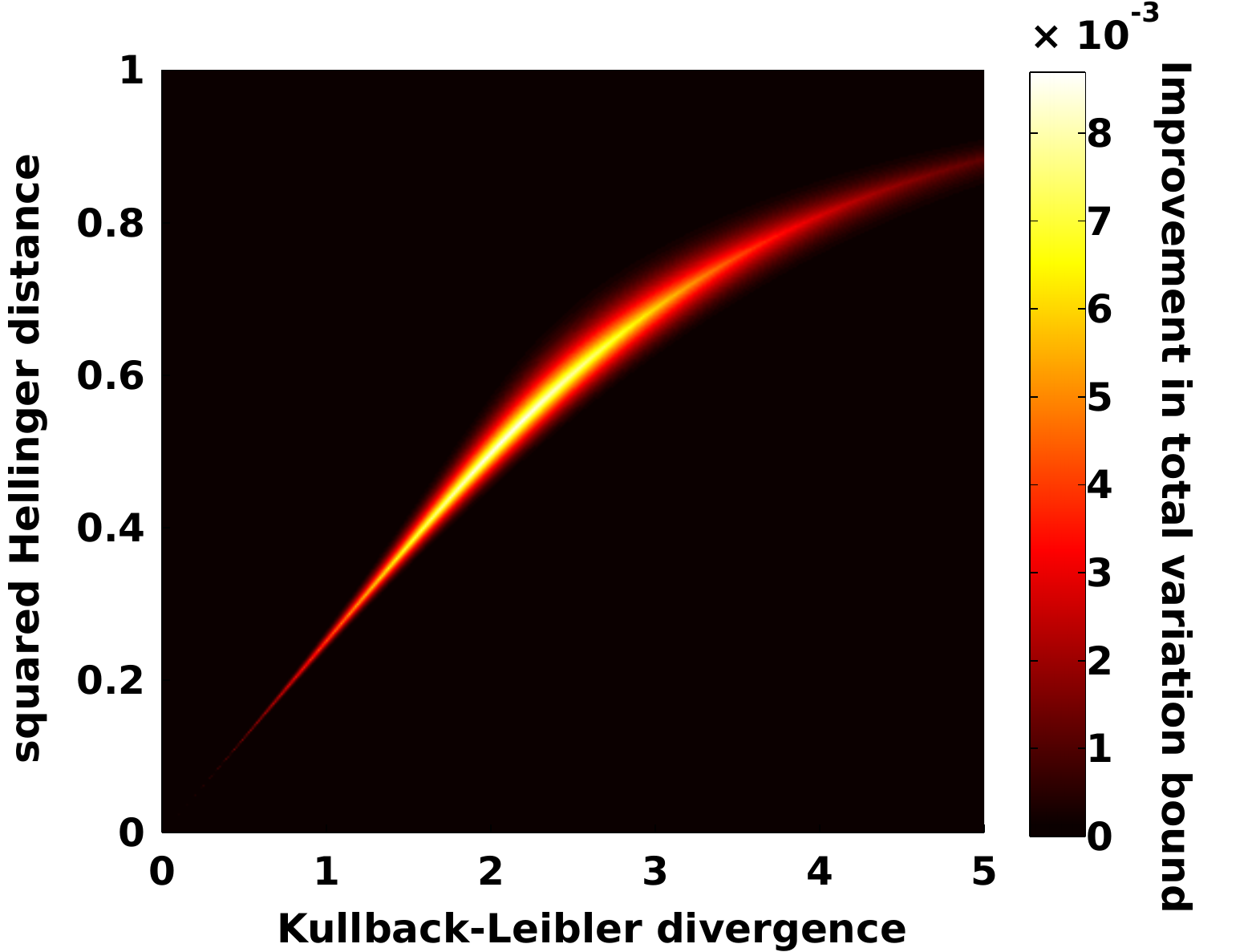}
\caption{Improvement over simple point-wise minimum of
    single-coordinate bounds. The color of the pixel at $(H,K)$
     represents the magnitude of the  left hand side of~\eqref{inform}. The
  bright region corresponds to $(H, K)$ for which the bound
  displayed in Figure~\ref{twoConstraints} is a strict improvement
  over the simple pointwise minimum of the two bounds shown in Figure
  \ref{fig:simple}.} 
\label{improvement}
\end{figure}

\subsection{The General Case}\label{heer}
Theorem~\ref{exact} requires $D_f$ to be a primitive divergence. We
do not know if, in general, the optimization problems~\eqref{finiteA}
and~\eqref{finiteB} can be solved by convex optimization
algorithms. However, if $m$ is not too large, heuristic optimization
techniques can be used. We demonstrate this in this subsection for $m
= 1$. 
\begin{example}\label{ex:hkl}
Consider the optimization problem~\eqref{finiteA} for $m=1$, $f(x) =
(\sqrt{x} - 1)^2/2$ and $f_1(x) = x \log x$. In other words, we
consider the problem of maximizing the squared Hellinger distance
subject to an upper bound on the Kullback-Leibler divergence. The
optimization problem~\eqref{finiteA} is clearly 4-dimensional (there
are six variables in all $p_1, p_2, p_3$ and $q_1, q_2, q_3$ but they
satisfy two linear constraints as they sum to one). Because the
variable space is only 4-dimensional, there was no trouble solving
this by gridding the parameter space. We plot the solution in
Figure~\ref{fig:hkl}(a) where each blue dot shows $A(K) =:
A^H_{KL}(K)$ for a different value of $K$. 

The quantity $A^H_{KL}(K)$ can be used to better understand the
inequality~\eqref{inform}. Indeed, when we overlay the curve $(K,
A_{KL}^H(K))$ on Figure~\ref{improvement} (see
Figure~\ref{fig:hkl}(b)), we see that the curve $(K, A_{KL}^H(K))$
(plotted by the blue line) lies above the region where the
inequality~\eqref{inform} holds. Only the constraint on $D_{KL}(P||Q)$ is
active in the optimization problem considered in Example~\ref{tvhkl}
when $H > A_{KL}^H(K)$.  For such $(H,K)$, therefore, the
inequality~\eqref{inform} does not hold. 

\begin{figure}
	\centering
	\begin{subfigure}[b]{\textwidth}
		\centering
		\quad \includegraphics[width =4.5 in ]{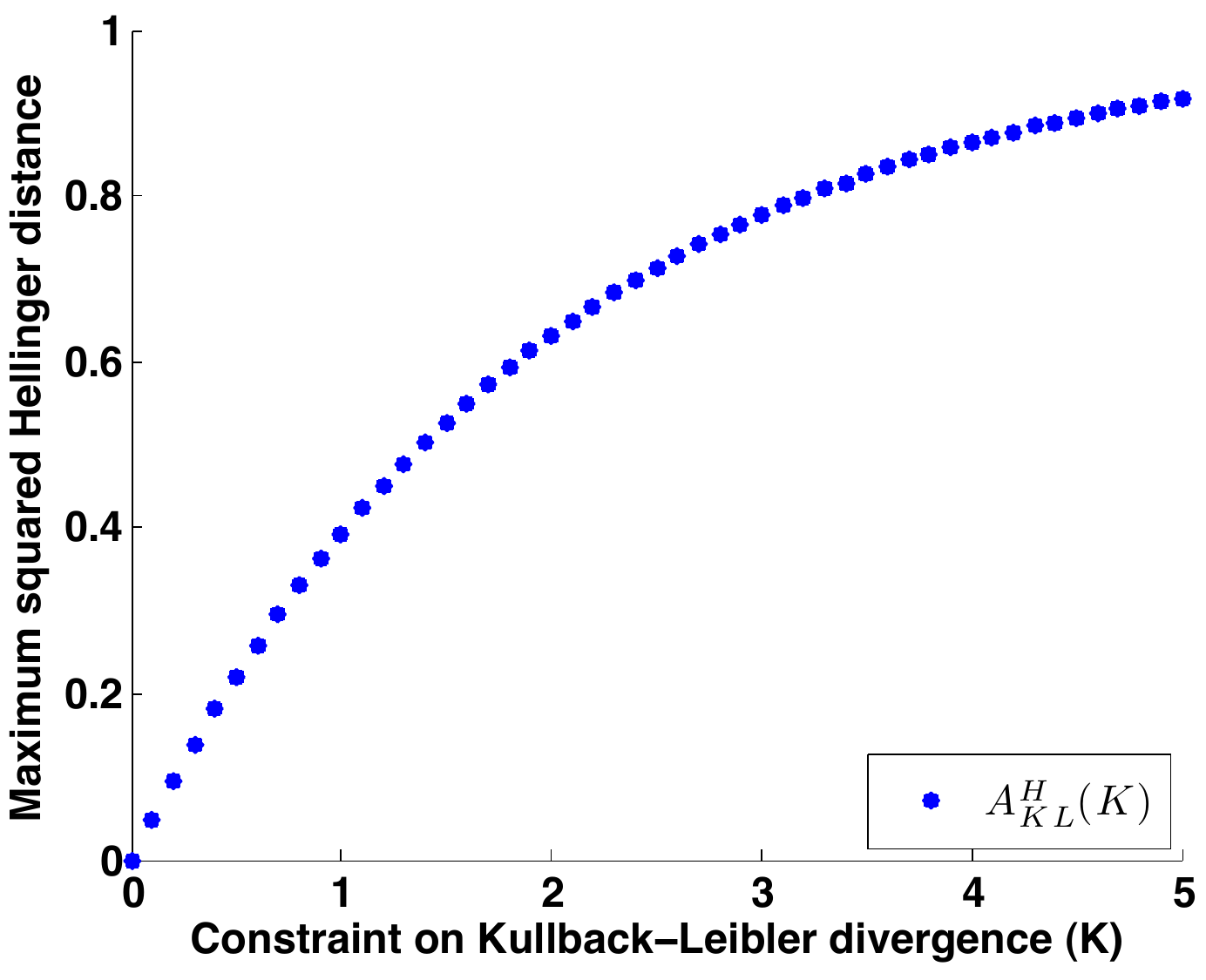}
	\end{subfigure}\\[1em]
	\begin{subfigure}[b]{\textwidth}   
		\centering  
		\quad \includegraphics[width=4.5 in]{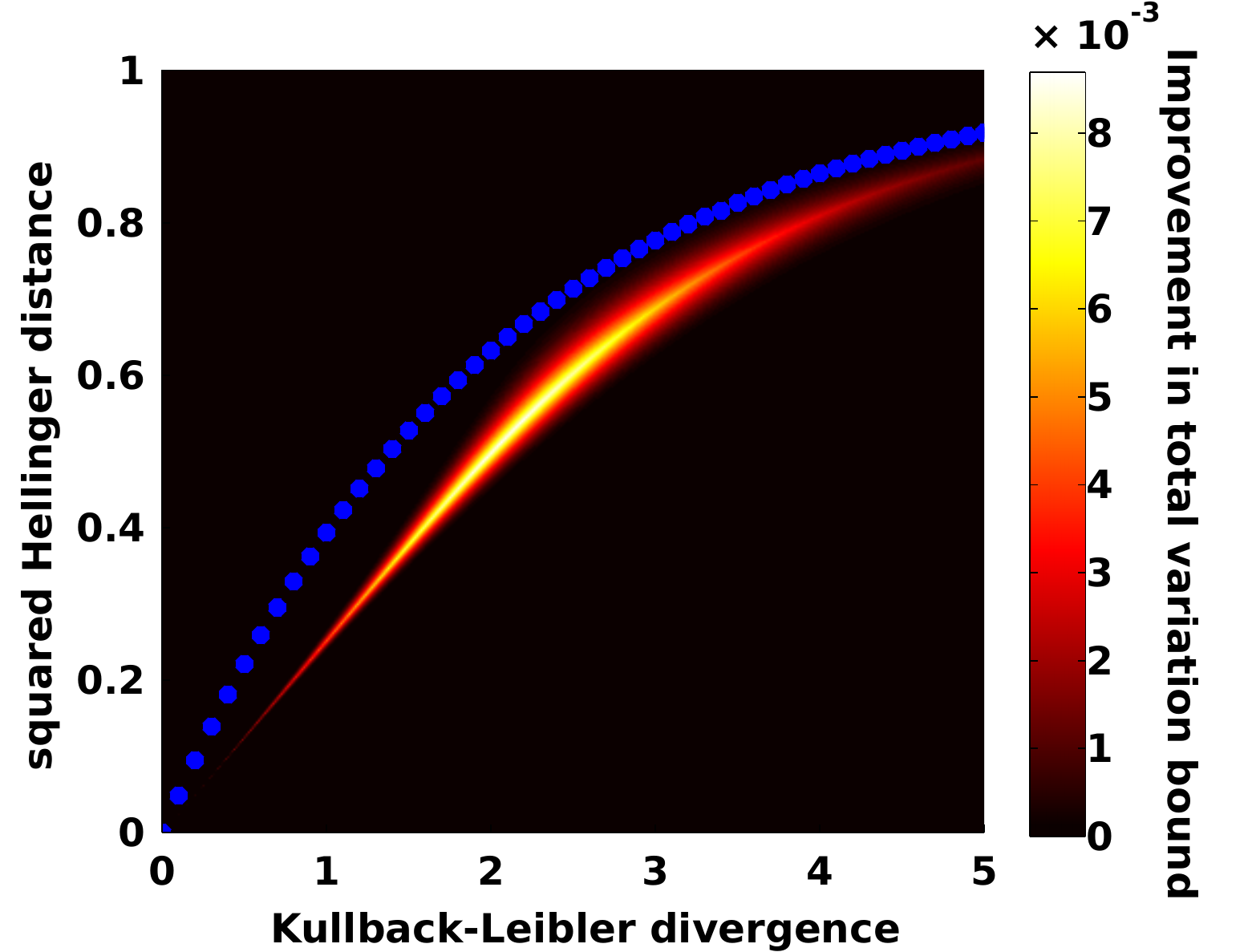}
	\end{subfigure}
       \caption{A sharp inequality between squared Hellinger
           distance and Kullback-Leibler divergence bounds the support
           of the ridge. The upper panel displays a sharp inequality
         between squared Hellinger and Kullback-Leibler divergence.
         The height of each blue dot represents the optimal value
         $A^H_{KL}(K)$ with a different constraint, $K$, on the
         Kullback-Leibler divergence. The 
         lower panel shows the same blue curve overlaid on Figure
         \ref{improvement}.  Observe that the region with positive
         improvement is bounded by the blue curve from the upper panel.} 
       \label{fig:hkl}
\end{figure}
\end{example}

\begin{example}\label{numtight}
Consider maximizing the squared Hellinger distance between $P$ and $Q$
with the total variation between $P$ and $Q$, $V(P,Q)$, bounded by
$V$. In other words, we consider the special case of the
problem~\eqref{finiteA} for $m = 1$, $f(x) = (\sqrt{x} - 1)^2/2$ and
$f_1(x) = |x-1|/2$. This is a special case of the problem we
considered in section \ref{tightness} where we proved that $A_2(V) <
A_3(V)$ for all $V \in (0, 1)$. Here we confirm this fact numerically. 

We compute both the quantities $A_2(V)$ and $A_3(V)$ by a gridded
search over pairs of probabilities satisfying the constraint in
$\Ps_2$ and $\Ps_3$ respectively. These functions are plotted in
Figure~\ref{htv}. Each red triangle in Figure \ref{htv} shows $A_3(V)$
for a different $V$. Each point in the dotted blue line shows $A_2(V)$
for a different $V$.  It is evident that the inequality $A_2(V) <
A_3(V)$ holds for all $V \in (0, 1)$. In other words, when we restrict the
constraint set to probability measures in $\Ps_2$, the maximum
Hellinger distance is strictly smaller for all $V \in (0,
1)$. Therefore, Theorem~\ref{main} is in general tight and cannot be
improved. 

Note also that the plot $A_3(V)$ agrees with the form $A_3(V) = A(V) =
V(f(0) + f'(\infty)) = V$ derived in Section~\ref{tightness}. This
gives rise to the sharp inequality $H^2(P, Q) \leq V(P, Q)$ which is
again attributed to~\cite{LeCam:86book}. 

\begin{figure}[h]
\centering
\includegraphics[width =4.5 in ]{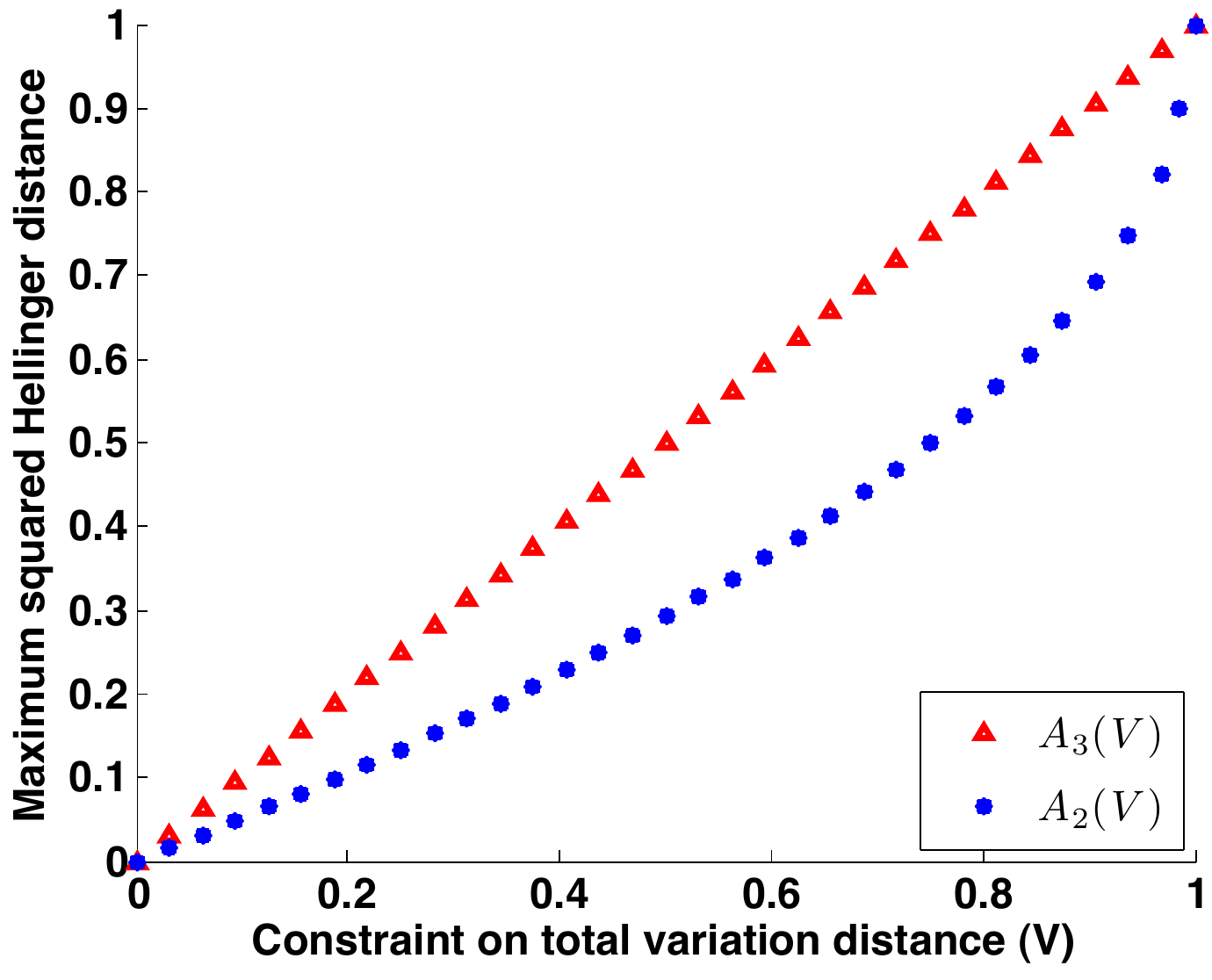}
\caption{Three point measures strictly improve on two point measures.
  Each red triangle
  shows $A_3(V)$ computed by a gridded search over pairs of
  probability measures in $\Ps_3$. Each blue dot shows $A_2(V)$
  computed by a gridded search over pairs of probability measures in
  $\Ps_2$. The simulation over three point measures is exactly a
 straight line with slope one--agreeing with Le Cam's bound $H^2 \le
 V$.  And $A_2(V) < A_3(V)$ for all $V\in (0,1)$.}  
\label{htv}
\end{figure}
\end{example}

\begin{example}
The capacitory discrimination between two probability measures $P$ and
$Q$ is defined by 
\begin{equation*}
C(P, Q) = D_{KL}\left(P||\frac{P+Q}{2}\right) +
D_{KL}\left(Q||\frac{P+Q}{2}\right). 
\end{equation*}
It is easy to check that $C(P, Q)$ is an $f$-divergence that
corresponds to the convex function: 
\begin{equation}\label{capd}
 x \log x - (x+1) \log(x+1) + 2\log 2. 
\end{equation}
The triangular discrimination $\Delta(P, Q)$  is another
$f$-divergence that corresponds to the convex function
\begin{equation}\label{trd}
 \frac{(x-1)^2}{x+1}.
\end{equation}
Tops{\o}e proved the following inequality between these two
$f$-divergences \cite{TopsoeCapDiv}: 
\begin{equation}\label{TopIneq}
\frac 1 2 \Delta(P, Q) \le C(P, Q) \le (\log 2) \Delta (P, Q).
\end{equation}
Let us investigate here the sharpness of these inequalities. Let 
\begin{equation*}
  A(D_1) := \sup \left\{C(P, Q) : \Delta(P, Q) \leq D_1 \right\}
\end{equation*}
and 
\begin{equation*}
  B(D_1) := \inf \left\{C(P, Q) : \Delta(P, Q) \geq D_1 \right\}. 
\end{equation*}
We solved the optimization problems~\eqref{finiteA} and~\eqref{finiteB}
for $m = 1$, $f(x)$ given by~\eqref{capd} and $f_1(x)$ given
by~\eqref{trd} by a gridded search. The resulting solutions for
$A(D_1)$ and $B(D_1)$ are plotted in Figure~\ref{ABtopsoe}, with red
triangles corresponding to $A(D_1)$ and blue dots corresponding to
$B(D_1)$. We have also plotted the bounds given by~\eqref{TopIneq} in
Figure~\ref{ABtopsoe} with the green line corresponding to $(\log 2)
D_1$ and the blue line to $D_1/2$. It is clear from the figure that
the upper bound in~\eqref{TopIneq} is sharp while the lower bound is
not sharp. The sharp lower bound is given by $B(D_1)$. We are unaware
of an analytic formula for $B(D_1)$, but we conjecture that $B_2(D_1) =
B_3(D_1)$ because this equality holds numerically.  It may be possible
to use this fact to find an analytic formula for $B(D_1)$. 

\begin{figure}[h]
\centering
\includegraphics[width = 4.5 in ]{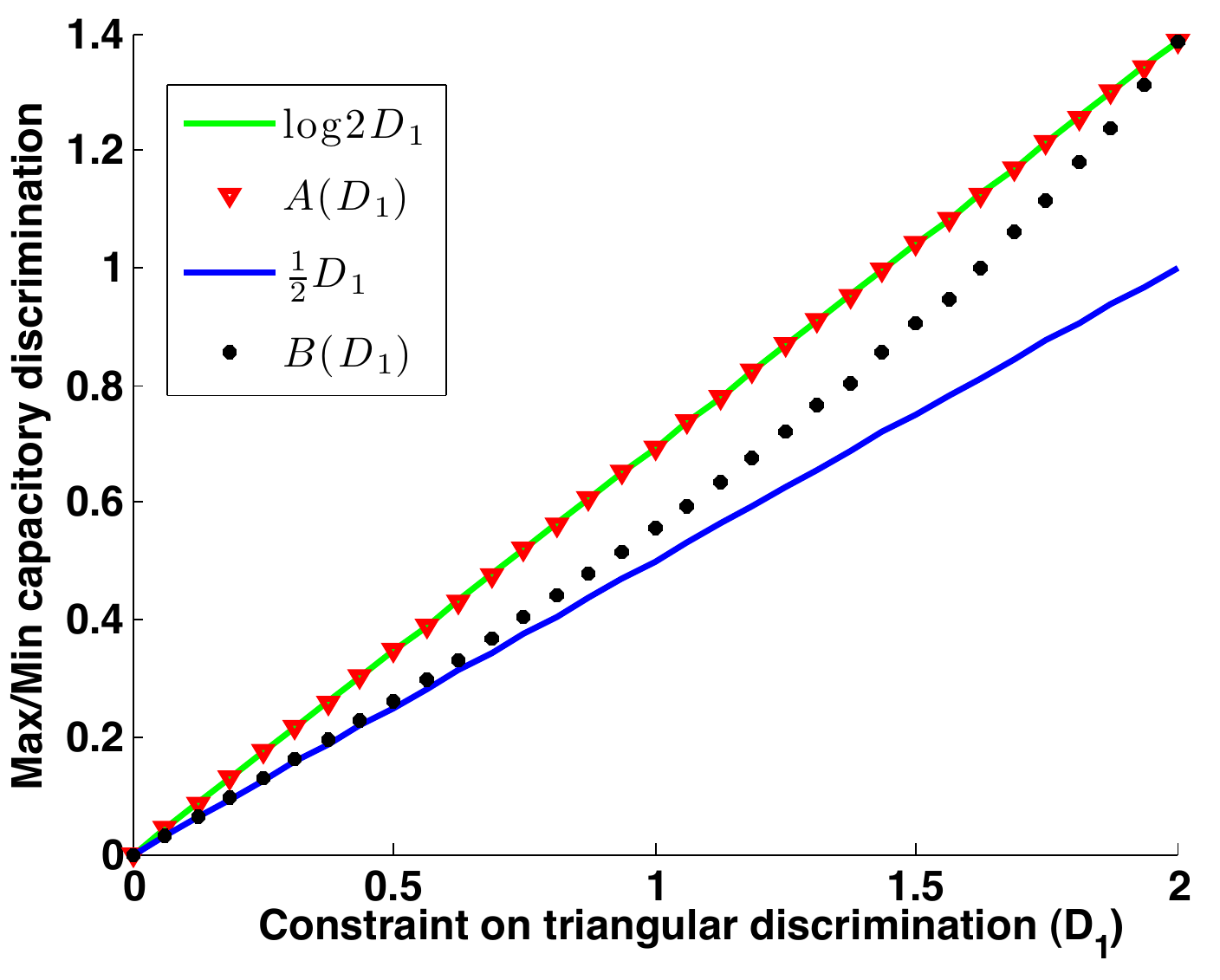}
\caption{The green line with slope $\log 2$ and the blue line with slope $\frac 1 2$ trace the bounds in \eqref{TopIneq}, while the red triangles and the black dots display $A(D_1)$ and $B(D_1)$ respectively.}
\label{ABtopsoe}
\end{figure}
\end{example}

\section*{Acknowledgment}
Adityanand Guntuboyina is extremely thankful to David Pollard for many
stimulating discussions. We are also grateful to an anonymous referee
for giving us Theorem 4.1 and its simple proof and for indicating that
it can be used to easily prove a weaker version of our main theorem.

\bibliographystyle{plainnat}
\bibliography{AG}

\def\noopsort#1{}
\begin{thebibliography}{34}
\providecommand{\natexlab}[1]{#1}
\providecommand{\url}[1]{\texttt{#1}}
\expandafter\ifx\csname urlstyle\endcsname\relax
  \providecommand{\doi}[1]{doi: #1}\else
  \providecommand{\doi}{doi: \begingroup \urlstyle{rm}\Url}\fi

\bibitem[Ali and Silvey(1966)]{AliSilvey}
S.~M Ali and S.~D Silvey.
\newblock A general class of coefficients of divergence of one distribution
  from another.
\newblock \emph{Journal of the Royal Statistical Society, Series B},
  28:\penalty0 131--142, 1966.

\bibitem[B{\'a}r{\'a}ny and Karasev(2012)]{BaranyKarasev}
I.~B{\'a}r{\'a}ny and R.~N. Karasev.
\newblock Notes about the {C}arath{\'e}odory number.
\newblock \emph{Discrete and Computational Geometry}, 48:\penalty0 783--792,
  2012.

\bibitem[Barron(1986)]{BarronCLT}
A.~Barron.
\newblock Entropy and the central limit theorem.
\newblock \emph{Annals of Probability}, 14:\penalty0 336--342, 1986.

\bibitem[Cover and Thomas(2006)]{CoverThomas}
T.~Cover and J.~Thomas.
\newblock \emph{Elements of Information Theory}.
\newblock Wiley, 2 edition, 2006.

\bibitem[Csisz{\'a}r(1966)]{Csiszar66}
I.~Csisz{\'a}r.
\newblock A note on {J}ensen's inequality.
\newblock \emph{Studia Scientarium Mathematicarum Hungarica}, 1:\penalty0
  185--188, 1966.

\bibitem[Csisz{\'a}r(1967{\natexlab{a}})]{Csiszar67}
I.~Csisz{\'a}r.
\newblock Information-type measures of difference of probability distributions
  and indirect observations.
\newblock \emph{Studia Scientarium Mathematicarum Hungarica}, 2:\penalty0
  299--318, 1967{\natexlab{a}}.

\bibitem[Csisz{\'a}r(1967{\natexlab{b}})]{Csiszar67fdiv}
I.~Csisz{\'a}r.
\newblock On topological properties of $f-$divergences.
\newblock \emph{Studia Scientarium Mathematicarum Hungarica}, 2:\penalty0
  329--339, 1967{\natexlab{b}}.

\bibitem[Csisz{\'a}r and Shields(2004)]{CsiszarShields}
I.~Csisz{\'a}r and P.~Shields.
\newblock Information theory and statistics: a tutorial.
\newblock \emph{Foundations and Trends in Communications and Information
  Theory}, 1:\penalty0 417--528, 2004.

\bibitem[Fedotov et~al.(2003)Fedotov, Harremo\"{e}s, and
  Tops{\o}e]{RefinePinsker}
Alexei Fedotov, Peter Harremo\"{e}s, and Flemming Tops{\o}e.
\newblock Refinements of pinsker's inequality.
\newblock \emph{IEEE Transactions on Information Theory}, 49:\penalty0
  1491--1498, 2003.

\bibitem[Gibbs and Su(2002)]{GibbsSu}
A.~L. Gibbs and F.~E. Su.
\newblock On choosing and bounding probability metrics.
\newblock \emph{International Statistical Review}, 70:\penalty0 419--435, 2002.

\bibitem[Gilardoni(2006)]{Gilardoni06}
Gustavo~L. Gilardoni.
\newblock On the minimum f-divergence for given total variation.
\newblock \emph{Comptes Rendus de l'Academie des Sciences, Paris, Ser.~I Math},
  343:\penalty0 763--766, 2006.

\bibitem[Guntuboyina(2011{\natexlab{a}})]{GuntuFdiv}
A.~Guntuboyina.
\newblock Lower bounds for the minimax risk using $f$ divergences, and
  applications.
\newblock \emph{IEEE Transactions on Information Theory}, 57:\penalty0
  2386--2399, 2011{\natexlab{a}}.

\bibitem[Guntuboyina(2011{\natexlab{b}})]{GuntuThesis}
Adityanand Guntuboyina.
\newblock \emph{Minimax Lower Bounds}.
\newblock PhD thesis, Yale University, 2011{\natexlab{b}}.

\bibitem[Harremo\"{e}s(2003)]{Harremoes}
P.~Harremo\"{e}s.
\newblock Convergence to the {P}oisson distribution in information divergence.
\newblock Technical report, University of Copenhagen, Copenhagen, Denmark,
  2003.
\newblock Preprint Series, No. 2.

\bibitem[Harremo\"{e}s and Vajda(2011)]{HarremoesVajda}
P.~Harremo\"{e}s and I.~Vajda.
\newblock On pairs of $f$-divergences and their joint range.
\newblock \emph{IEEE Transactions on Information Theory}, 57:\penalty0
  3230--3235, 2011.

\bibitem[Kemperman(1969)]{kemperman69}
J.~H.~B. Kemperman.
\newblock On the optimum rate of transmitting information.
\newblock In \emph{Probability and Information Theory}. Springer-Verlag, 1969.
\newblock Lecture Notes in Mathematics, 89, pages 126--169.

\bibitem[Kullback(1967)]{Kullback67IEEE}
S.~Kullback.
\newblock A lower bound for discrimination information in terms of variation.
\newblock \emph{IEEE Transactions on Information Theory}, 13:\penalty0
  126--127, 1967.

\bibitem[Le~Cam(1986)]{LeCam:86book}
L.~Le~Cam.
\newblock \emph{Asymptotic Methods in Statistical Decision Theory}.
\newblock Springer-Verlag, New York, 1986.

\bibitem[Liese(2012)]{Liese}
Friedrich Liese.
\newblock $\phi$-divergences, sufficiency, {B}ayes sufficiency, and deficiency.
\newblock \emph{Kybernetika}, 48:\penalty0 690--713, 2012.

\bibitem[Liese and Vajda(2006)]{LieseVajda}
Friedrich Liese and Igor Vajda.
\newblock On divergences and informations in statistics and information theory.
\newblock \emph{IEEE Transactions on Information Theory}, 52:\penalty0
  4394--4412, 2006.

\bibitem[Marton(1986)]{MartonBlowUp}
K.~Marton.
\newblock A simple proof of the blowing-up lemma.
\newblock \emph{IEEE transformations on {I}nformation theory}, 32:\penalty0
  445--446, 1986.

\bibitem[Marton(1996{\natexlab{a}})]{MartonAnnProb}
K.~Marton.
\newblock Bounding $\bar{d}$-distance by informational divergence: a method to
  prove measure concentration.
\newblock \emph{Annals of Probability}, 24:\penalty0 857--866,
  1996{\natexlab{a}}.

\bibitem[Marton(1996{\natexlab{b}})]{MartonGFA}
K.~Marton.
\newblock A measure concentration inequality for contracting {M}arkov chains.
\newblock \emph{Geometric and Functional Analysis}, 6:\penalty0 556--571,
  1996{\natexlab{b}}.

\bibitem[{\"O}sterreicher and Vajda(1993)]{OsterreicherVajda}
F.~{\"O}sterreicher and I.~Vajda.
\newblock Statistical information and discrimination.
\newblock \emph{IEEE Trans. Inform. Theory}, 39:\penalty0 1036--1039, 1993.

\bibitem[Phelps(1966)]{Phelps66Choquet}
Robert~R. Phelps.
\newblock \emph{Lectures on {C}hoquet's theorem}.
\newblock Van Nostrand, 1966.

\bibitem[Pinsker(1960)]{PinskerIneq}
M.~S Pinsker.
\newblock \emph{Information and information stability of random variables and
  processes}.
\newblock Izv. Akad. Nauk, Moscow, 1960.

\bibitem[Reid and Williamson(2011)]{ReidWilliamsonJMLR}
M.~D. Reid and R.~C. Williamson.
\newblock Information, divergence and risk for binary experiments.
\newblock \emph{Journal of Machine Learning Research}, 12:\penalty0 731--817,
  2011.

\bibitem[Reid and Williamson(2009)]{GenPin}
Mark~D. Reid and Robert~C. Williamson.
\newblock Generalized {P}insker inequalities.
\newblock In \emph{Proceedings of the 22nd Annual Conference on Learning
  Theory}, 2009.

\bibitem[Rudin(1991)]{Rudin.func}
Walter Rudin.
\newblock \emph{Functional Analysis}.
\newblock International Series in Pure and Applied Mathematics. McGraw-Hill
  Inc., second edition, 1991.

\bibitem[Tops{\o}e(1979)]{Topsoe79}
F.~Tops{\o}e.
\newblock Information theoretical optimization techniques.
\newblock \emph{Kybernetika}, 15:\penalty0 8--27, 1979.

\bibitem[Tops{\o}e(2000)]{TopsoeCapDiv}
F.~Tops{\o}e.
\newblock Some inequalities for information divergence and related measures of
  discrimination.
\newblock \emph{IEEE Trans. Inform. Theory}, 46:\penalty0 1602--1609, 2000.

\bibitem[Tsybakov(2009)]{Tsybakovbook}
Alexandre Tsybakov.
\newblock \emph{Introduction to Nonparametric Estimation}.
\newblock Springer-Verlag, 2009.

\bibitem[Vajda(1970)]{Vajda70}
I.~Vajda.
\newblock Note on discrimination information and variation.
\newblock \emph{IEEE Trans. Inform. Theory}, 16:\penalty0 771--773, 1970.

\bibitem[Yu(1997)]{Yu97lecam}
Bin Yu.
\newblock Assouad, {F}ano, and {Le Cam}.
\newblock In D.~Pollard, E.~Torgersen, and G.~L. Yang, editors,
  \emph{Festschrift for Lucien Le~Cam: Research Papers in Probability and
  Statistics}, pages 423--435. Springer-Verlag, New York, 1997.

\end{thebibliography}

\end{document}